\documentclass[10 pt,reqno]{amsart}
\usepackage{amssymb,amsmath,amstext,amsgen,amsfonts,amsthm,fullpage,enumerate,mathrsfs,verbatim,url,hyperref}
\usepackage[usenames,dvipsnames]{color}
\usepackage{ctable}
\usepackage[nobysame,alphabetic]{amsrefs}

\newcommand{\cS}{\mathcal{S}}

\newcommand{\largeconst}{{R}}
\newcommand{\ourpower}{{r}}
\newcommand{\smallconst}{{\frac{4-\ourpower}{8}}}%\frac{p-\ourpower}{4}

\DeclareMathOperator{\rad}{rad}
\DeclareMathOperator{\diam}{diam}
\DeclareMathOperator{\cent}{Center}

\newcommand{\FFP}{{F\!F\!P}}
\newcommand{\N}{\mathbb{N}}
\newcommand{\R}{\mathbb{R}}

\newcommand{\D}{D}
\renewcommand{\H}{\mathbb{H}}

\newcommand{\cB}{\mathcal{B}}
\newcommand{\cH}{\mathcal{H}}

\newcommand{\ball}{B}
\newcommand{\dist}{d}

\newcommand{\Kset}{{E}}

\renewcommand{\epsilon}{\varepsilon}
\renewcommand{\phi}{\varphi}

\newtheorem{theorem}{Theorem}[section]
\newtheorem{proposition}[theorem]{Proposition}
\newtheorem{lemma}[theorem]{Lemma}

\newtheorem{maintheorem}{Theorem}

\theoremstyle{remark}
\newtheorem{remark}{Remark}
\begin{document}

\title{AN UPPER BOUND FOR THE LENGTH OF A TRAVELING SALESMAN PATH IN THE HEISENBERG GROUP}
\author{Sean Li \and Raanan Schul}
\thanks{
S.~Li was partially supported by a postdoctoral research fellowship 
NSF DMS-1303910.
R.~ Schul was partially supported by a fellowship from the  Alfred P. Sloan Foundation as well as  by NSF  DMS 11-00008.}
\date{\today}
\address{Department of Mathematics,  University of Chicago, 5734 S University Avenue
Chicago, IL 60637}
\email{seanli@math.uchicago.edu}
\address{Department of Mathematics, Stony Brook University, Stony Brook, NY 11794-3651}
\email{schul@math.sunysb.edu}

\begin{abstract} 
We show that
a sufficient condition for a subset $\Kset$ in the Heisenberg group (endowed with the 
Carnot-Carath\'{e}odory metric) to be  contained in a rectifiable curve is that  it satisfies a modified analogue of Peter Jones's geometric lemma.  
Our estimates improve on those of  \cite{FFP},  
by replacing the power $2$ of the Jones-$\beta$-number with any power  $\ourpower<4$.
%allowing for any power $\ourpower<4$ to replace the power $2$ of the Jones-$\beta$-number. 
This complements (in an open ended way) our work  \cite{Li-Schul-beta-leq-length}, where we showed that such an estimate was necessary, but with  $\ourpower=4$.  
\end{abstract}

\maketitle

%\tableofcontents

\section{Introduction}

Let $\H$ denote the Heisenberg group, endowed with the Carnot-Carath\'{e}odory distance and $\Kset \subseteq \H$ be any subset. Let $B(x,t)$ be the (closed) ball of radius $t$ around$x$.  Then we denote
\begin{align*}
 \beta_{\Kset,\H}(B)=\inf\limits_{L} \sup\limits_{x\in \Kset\cap B} \frac{\dist(x,L)}{\diam(B)},
\end{align*}
with the infimum is taken over all {\it horizontal lines} $L$ (to be defined in the next section). 
If it is clear from the context which set $E$ we are referring to, we will omit it from the notation and write $\beta_{\H}(B)$

In this paper we prove the following theorem.
\begin{maintheorem} \label{t:maintheorem-INT}
Let $\ourpower<4$ be fixed.
There is a constant $C=C(\ourpower)>0$ such that for any set $\Kset \subseteq \H$ if
\begin{align*}
  \diam(\Kset) + \int\limits_{\H} \int\limits_0^{+\infty} \beta_{\Kset,\H}(B(x,t))^{\ourpower}\frac{dt}{t^4}d\cH^4(x) < \infty,
\end{align*}
then there exists a rectifiable curve $\Gamma \supset \Kset$
such that
\begin{align}\label{e:length-beta-upper-bound-INT}
\cH^1(\Gamma)
\leq C\left(\diam(\Kset) + \int\limits_{\H} \int\limits_0^{+\infty} \beta_{E,\H}(B(x,t))^{\ourpower}\frac{dt}{t^4}d\cH^4(x)\right).
\end{align}
\end{maintheorem}

We do not know if Theorem \ref{t:maintheorem-INT} is still true when $r = 4$, but we do know that it is almost optimal:  in \cite{Li-Schul-beta-leq-length} we showed the following.
\begin{maintheorem} \label{t:beta-leq-length}
There is a constant $C>0$ such that for any  rectifiable curve $\Gamma$
the following holds.
We have 
\begin{align}
\int\limits_{\H} \int\limits_0^{+\infty} \beta_{\Gamma,\H} (B(x,t))^4\frac{dt}{t^4}d\cH^4(x) 
\leq C\cH^1(\Gamma).  \label{e:sum-beta-upper-bound-INT}
\end{align}
\end{maintheorem}

%\SL{I don't know if this is worth saying considering we don't know if there is actually a discrepancy: This discrepancy between $r<4$ and the endpoint $r=4$ is quite different than other instances where similar theorems are known.}

%We remark that 

%\begin{remark}
The use of Hausdorff measure of dimension 4 directly corresponds to the Hausdorff dimension of $\H$ and the  power of $t$.  However, it does not correspond to the power 4 of $\beta$.  That 4 comes from the modulus of curvature coming directly from the Heisenberg geometry.  
In an $n$-dimensional Euclidean space,
the analogous theorems hold with the same power $r=2$ for both Theorem \ref{t:maintheorem-INT} and \ref{t:beta-leq-length}, and the power of $t$ as well as the Hausdorff measure dimension are $n$ \cite{Jones-TSP,  Ok-TSP}.
%\end{remark}
%($t^4$ is replaced by $t^n$ as $n$ is the Hausdorff dimension of $\R^n$; affine lines replace horizontal lines) \cite{Jones-TSP, Ok-TSP}.   
If one discretizes the integral to a sum in an appropriate way, the same holds for an infinite dimensional Hilbert space \cite{Schul-TSP}, again, with $r=2$.  In a general metric space, with no assumptions on the set $\Kset$, there is an analogue of Theorem \ref{t:maintheorem-INT} \cite{Ha-non-AR}, however the analogue of Theorem \ref{t:beta-leq-length} is false in that setting \cite{Schul-survey}.  If one adds the assumption that $\Kset$ is $1$-Ahlfors-regular (i.e. it supports a measure with linear upper and lower bounds on its growth) 
then, in this general metric setting analogues of both  Theorem \ref{t:maintheorem-INT} and \ref{t:beta-leq-length} hold \cite{Ha-AR, Schul-AR}. Finally we note  that extensive work has been done on the Euclidean case, where one approximates with $k$-planes instead of  lines.

In Euclidean space, there is a very deep connection between  how well $k$-dimensional sets are approximated by $k$-planes (in a sense analogous to the right hand side of \ref{t:beta-leq-length}, with $r=2$ and appropriately scaled) and singular integral operators. See for example \cite{Pajot-book, DS91, DS93, Tolsa} and references therein.  The heuristic point (that has  rigorous meaning as well) is that the multi-scale approximations by $k$-planes are analogous to a multi-scale decomposition of a function (or its derivative) into a wavelet basis and is a measurement of how fast one approaches a tangent.  In a ball, one  approximates a singular integral against a measure supported on a set by a similar singular integral, against a measure supported on an appropriate  $k$-dimensional hyperplane, and sums over balls at all scales and locations.  It would be interesting to explore the connection between  our results and singular integrals on one dimensional subsets of the Heisenberg group.    

Naturally, the case  $k=1$ allows more results in Euclidean space, and this is the case which makes sense in the Heisenberg group due to the lack of rectifiable surfaces. One may view our results as answering the question `Give necessary and sufficient conditions for a collection of sites to be visited by a curve of finite length in the Heisenberg group'.  Once can ask further questions like `Can you connect the sites by a finite length curve such that movement between the different sites is efficient' as was done in \cite{AS-shortcuts}, where it was shown that any rectifiable curve $\Gamma$  is contained inside a quasiconvex rectifiable curve $\tilde\Gamma$ of comparable length. Such questions, if generalized appropriately, could end up having applications.

The curve in Theorem \ref{t:maintheorem-INT} may be  constructed in an algorithmic way.  In fact, up to a natural modification of the constants used (and the metric), this curve is constructed in \cite{FFP}. The authors there get the estimate \eqref{e:length-beta-upper-bound-INT} for $\ourpower\leq 2$. 
Several years after \cite{FFP} was published, an example was constructed of a finite length curve such that the right hand side of \eqref{e:length-beta-upper-bound-INT}, with $\ourpower\leq2$, is infinite \cite{juillet}.
With a minor modification (taking the parameters of the construction to be $\theta_k = \frac{c}{k^p}$ for $p > 1/2$ instead of $p = 1$), this example remains a valid counterexample for all $\ourpower<4$ and so suggests that one may be able to improve on \cite{FFP}.
Indeed our contribution in this note  is to improve upon the estimates of \cite{FFP} in a non-trivial fashion, and as a consequence to extend the range of valid $\ourpower$ by adding the range $(2,4)$.  The fact that one may improve upon the power $\ourpower=2$ came from \cite{li-carnot}, which gave a parametric version of Theorem \ref{t:beta-leq-length} with power different from 2.  In the same paper, it was also shown that the parametric analogue of $\ourpower$ is related to the Markov convexity of $\H$, which was recently calculated to be 4 in \cite{li-heisenberg}.

Previous proofs of statements like Theorem \ref{t:maintheorem-INT} used the farthest insertion algorithm along with a curvature inequality of the form
\begin{align}
  d(a,b) + d(b,c) - d(a,c) \leq C \beta(B)^p \diam(B) \label{e:curvature}
\end{align}
where $B$ is some ball of radius $r$ and $a,b,c \subset E \cap B$ are points whose mutual distances are all comparable to $r$.  Here, $\beta(B)$ depends on the context of the problem.

The curve is then built in an iterative fashion as  in \cite{Jones-TSP}.  Letting $\Delta_k$ be a $2^{-k}$-net of $\Kset$ that was connected by a polygonal curve $\Gamma_k$, one builds $\Gamma_{k+1}$ by adding and deleting segments of $\Gamma_k$ in a clever way to include $\Delta_{k+1}$.  
The  inequality  \eqref{e:curvature} enables one to bound the sum of the telescoping series $\ell(\Gamma_{k+1}) - \ell(\Gamma_k)$ to be bounded by the Carleson sum of the $\beta$-numbers.  The curve $\Gamma$  is the  limit (along a subsequence) of the curves $\Gamma_j$, ensured by Arzel\`a-Ascoli.
For a detailed account of this see the appendix of \cite{FFP}, which also explains why we freely confuse a connected set of finite length  with a parametrized curve.

In \cite{FFP}, the authors showed that the $r = 2$ version of Theorem \ref{t:maintheorem-INT} is true by proving the inequality
\begin{align}
  d(a,b) + d(b,c) - d(a,c) \leq C \beta_\H(B)^2 \diam(B) \label{e:FFP-curvature}
\end{align}
and then applying the farthest insertion algorithm with \eqref{e:FFP-curvature}.  Here, $a,b,c$ are points in $\Kset \cap B$ that are spread out far enough.  Thus, the first obvious attempt would be to improve this inequality to get a higher power of the $\beta_\H(B)$.  However, this is impossible.  We discuss the geometry of the Heisenberg group in the next section, so it may be useful to review that before reading on.

Consider the points $a = (-1,0,0)$, $b = (0,0,\epsilon)$, $c = (1,0,0)$.  It is straightforward to calculate that there exists some absolute constant $c_1 > 0$ so that
\begin{align*}
  d(a,b) + d(b,c) - d(a,c) \geq c_1 \epsilon^{2}.
\end{align*}
However, one can see that if we let $L = \left\{\left( \left(1- \frac{\epsilon}{2}\right)t, \frac{\epsilon}{2} t,0\right) : t \in \R\right\}$ be a horizontal line, then there exists some absolute constant $c_2 > 0$ so that
\begin{align*}
  \max\{ d(a,L), d(b,L), d(c,L) \} \leq c_2 \epsilon.
\end{align*}
This is because if we set $a' = \left(1 - \frac{\epsilon}{2}, \frac{\epsilon}{2}, 0\right), c' = \left( \frac{\epsilon}{2} - 1, -\frac{\epsilon}{2},0\right)$ to be points in $L$, then $\max\{ d(a,a'), d(c,c') \} < c_2 \epsilon$.  Thus, by taking $\epsilon \to 0$, we see that the power 2 in \eqref{e:FFP-curvature} cannot be improved, at least in this generality.

What we will show is that for any $p < 4$, ball $B \subset \H$, and points $a,b,c \in \Kset \cap B$ that are well spread out, if $\Kset \cap B$ is sufficiently connected (the exact condition is given in the assumptions of Proposition \ref{p:future-balls}), there is some other ball $F_1(B) \subset \H$ for which
\begin{align}
  d(a,b) + d(b,c) - d(a,c) \leq C_p \beta_\H(F_1(B))^p \diam(F_1(B)). \label{e:future-balls-1}
\end{align}
This ball $F_1(B)$ may be smaller than $B$ and may not actually be contained in $B$, but it cannot be too much smaller and cannot be too far away from $B$. In fact, the radius of $F_1(B)$ is controlled from below by a function of $\beta_\H(B)$ (times that of $B$), and its distance from $B$ is a multiple of the diameter of $B$. Thus, we can control how many times each ball $B'$ is used as a $F_1(B)$.  In the case when $\Kset$ is not sufficiently connected enough to use \eqref{e:future-balls-1}, then the disconnection will enable us to use an accounting trick to show that \eqref{e:FFP-curvature} is sufficient to handle this case.  This search for balls with better $\beta$ properties deviates from the proofs of previous versions of Theorem \ref{t:maintheorem-INT} and is what enables us to improve on the $\ourpower = 2$ power despite not being able to improve upon \eqref{e:FFP-curvature}.

An obvious question is whether Theorem \ref{t:maintheorem-INT} is true for $r = 4$.  The fact that $r < 4$ is crucially used in two parts of the construction of this paper.  In finding $F_1(B)$, one has to do iterative searches for balls with progressively better properties.  That $r < 4$ guarantees that we only need to do $O(1)$ searches.  This allows us to control the constant in \eqref{e:future-balls-1}.  
The condition  $r < 4$ is also used to control the number of balls
$\{B': F_1(B')=B\}$ for each ball $B$.  We get that this number depends (linearly) on the logarithm of  $\beta_\H(B)$.   These logarithmic factors initially come into the estimates as multiplicative factors, but having  $r < 4$ allows us to remove these log terms by first proving the statement for some power $r' \in (r,4)$ and then relaxing the power to $r$.

Thus, we leave unanswered the question {\it Is there a constant $C=C(4)$ so that \eqref{e:length-beta-upper-bound-INT} holds for $\ourpower=4$?}

In the next section, we review some facts about the Heisenberg group and establish the notation for the rest of the paper.  In section 3, we will prove the lemmas needed for our main proposition, the proof of which will be given in section 4.  In section 5, we show how to adapt the construction of \cite{FFP} to use the proposition of Section 4 to prove Theorem \ref{t:maintheorem-INT}.

\section{Preliminaries}
Following \cite{FFP}, we will say that the Heisenberg group is the 3-dimensional Lie group $\H = (\R^3,\cdot)$ where the group multiplication is
\begin{align*}
  (x,y,z) \cdot (x',y',z') = \left(x + x', y+y', z+z' + 2 (xy' - x'y) \right).
\end{align*}
One sees then that $(0,0,0)$ is still the identity and we will refer to it as 0.

There is a natural path metric on $\H$ that we now describe.  Using the fact that group multiplication is smooth, we can define $H\H$, a left-invariant subbundle of the tangent bundle $T\H = T\R^3$ so that $H_0\H$ is the $xy$-plane in $\R^3$ and $H_g\H = (L_g)_*H_0\H$ for $g \in \H$ where $L_g$ is the smooth $\H \to \H$ map that is left multiplcation by $g$.  We can similarly use $L_g$ to endow $H\H$ with a left-invariant field of inner products $\{\langle \cdot, \cdot \rangle_g\}_{g \in \H}$.  The normalization is usually that the $x$ and $y$ unit vectors of $H_0\H$ are orthogonal under $\langle \cdot,\cdot \rangle_0$, but this is not too important.  Given $x,y \in \H$, we can now define the Carnot-Carath\'eodory distance between $x$ and $y$ to be
\begin{align*}
  d_{cc}(x,y) := \inf \left\{ \int_a^b \langle \gamma'(t),\gamma'(t) \rangle_{\gamma(t)}^{1/2} ~dt : \gamma \in C^1([a,b];\H), \gamma(a) = x, \gamma(b) = y, \gamma'(t) \in \Delta_{\gamma(t)} \right\}.
\end{align*}
A natural question to ask is whether the set of curves in the definition of the Carnot-Carath\'eodory metric is always nonempty given any two points $a,b \in \H$.  It is known that such $C^1$ curves always exist in $\H$ (see {\it e.g.} \cite{montgomery}).  Continuous paths that satisfy $\gamma'(t) \in \Delta_{\gamma(t)}$ (almost everywhere) are called {\it horizontal paths}.  As we are taking the Riemannian distance in a subclass of paths connecting two points, this geometry is sometimes called subriemannian geometry.

It is well known that a horizontal curve $\gamma = (\gamma_x,\gamma_y,\gamma_z) : I \to \H$ satisfies the following identity
\begin{align*}
  \gamma_z(b) - \gamma_z(a) - 2 \left( \gamma_x(a) \gamma_y(b) - \gamma_x(b) \gamma_y(a) \right) = 2 \int_a^b (\gamma_x(t) \gamma_y'(t) - \gamma_x'(t) \gamma_y(t)) ~dt, \qquad \forall a,b \in I.
\end{align*}
Thus, the change in $z$-coordinate of a horizontal curve as viewed in the group product is equal to four times the algebraic area swept by $(\gamma_x,\gamma_y)$ when viewed as a curve in $\R^2$.

While the Carnot-Carath\'eodory metric is a well defined path metric, it is not so easy to explicitly compute Carnot-Carath\'eodory distances between points.  Instead, we will work with the equivalent Koranyi metric (or Koranyi distance), for which distances are easier to compute.  We define the Koranyi norm as
\begin{align*}
  N : \H &\to \R \\
  (x,y,z) &\mapsto ((x^2 + y^2)^2 + z^2)^{1/4}.
\end{align*}
It then defines a left-invariant metric on $\H$ via $d(g,h) = N(g^{-1}h)$.  It is well known that this is a metric ({\it i.e.} satisfies the triangle inequality) and is biLipschitz equivalent to the Carnot-Carath\'eodory metric \cite{cygan}.  We have that $d(x,y) \leq d_{cc}(x,y)$ as $d(x,y) = d_{cc}(x,y)$ whenever $x$ and $y$ are on a horizontal line and so the Koranyi norm does not decrease length in horizontal paths.  For simplicity, we will assume the non-sharp (see (1.4) in \cite{balogh2006lifts}) lower bound
\begin{align}
  d(x,y) \geq \frac{1}{2} d_{cc}(x,y). \label{e:cc-lower-bound}
\end{align}

An important property of $\H$ is that it admits a family of dilation automorphisms.  Specifically, for each $\lambda > 0$, we can define the automorphism
\begin{align*}
  \delta_\lambda : \H &\to \H \\
  (x,y,z) &\mapsto (\lambda x, \lambda y, \lambda^2 z).
\end{align*}
These dilations scale the metric, that is, $d_{cc}(\delta_\lambda(x),\delta_\lambda(y)) = \lambda d_{cc}(x,y)$.  This can be verified by looking at the Jacobian of $\delta_\lambda$ and the remembering how the Carnot-Carath\'eodory metric is defined.  It is immediately verified by looking at the expression of the Koranyi norm that $\delta_\lambda$ also scales the Koranyi metric.

Rotations around the $z$-axis comprise a set of isometric automorphisms of $\H$.  This follows from looking at the formulas of the Koranyi norm and group multiplication and seeing that $xy' - x'y$ is just a cross product, which is invariant under rotations.

The Heisenberg group is known to be geometrically doubling.  That is, there exists some $M > 0$ so that any ball $B(x,r) \subset \H$ can be covered by $M$ balls of radius $\frac{r}{2}$.  The point is that $M$ can be chosen uniformly for all $x \in \H$ and $r > 0$.  This follows from the fact that the Lebesgue measure on $\R^3$ a Haar measure of $\H$ and balls of $\H$ grow like $r^4$.  A standard volume packing argument then gives that $\H$ is geometrically doubling.  That the Lebesgue measure on $\R^3$ is a Haar measure for $\H$ can be seen from the fact that the linear parts of the affine transforms in $\R^3$ that correspond to group translations in $\H$ are volume preserving.  The growth of balls comes from the anisotropic scaling of the dilations.

Another important feature of the Heisenberg group is the existence of a distinguished family of curves called {\it horizontal lines}.  Before we define horizontal lines, we first define horizontal elements of $\H$.  An element $g \in \H$ is said to be horizontal if $g$ lies on the $xy$-plane.  For such horizontal elements, we can extend $\delta_\lambda$ to all $\lambda \in \R$ to get $\delta_t(x,y,0) = (tx,ty,0)$.  The horizontal lines are subsets of the form $\{g \cdot \delta_t(h) : t \in \R\}$ where $g,h \in \H$ and $h$ is horizontal.  Horizontal lines essentially amount to horizontal curves where the tangent vector stays the same.  Note that the set of horizontal lines going through a specified point in $\H$ spans two dimensions instead of three as in $\R^3$.  This shows that most pairs of points in $\H$ cannot be joined by a horizontal line, which reflects a crucial difference between Heisenberg and Euclidean geometry.

We define the homomorphic projection
\begin{align*}
  \pi : \H &\to \R^2 \\
  (x,y,z) &\mapsto (x,y).
\end{align*}
It is immediately verifiable that this is 1-Lipschitz by looking at the Koranyi norm.  We also define the maps
\begin{align*}
  \tilde{\pi} : \H &\to \H \\
  (x,y,z) &\mapsto (x,y,0)
\end{align*}
and 
\begin{align*}
  NH : \H &\to \H \\
  g &\mapsto N(g^{-1} \tilde{\pi}(g)).
\end{align*}
Note that $\tilde{\pi}$ is not a homomorphism.  The map $NH$ gives a measurement of how nonhorizontal an element $g \in \H$ is by computing the distance of $g$ to the horizontal element ``below'' it.

Given subsets $K \subseteq L$ of any metric space $(X,d_X)$ and $\delta > 0$, we say $K$ is 
{\it $\delta$-connected}
 in $L$ if for each $x,y \in K$, there exists a finite sequence $\{z_i\}_{i=1}^n \subset L$ for which $z_1 = x$ and $z_n = y$ so that $d_X(z_i,z_{i+1}) < \delta$.

Finally, we recall that a  curve is a continuous map $\gamma$ whose domain is  an interval $I\subset \R$. If  $\gamma$ has finite arclength then its image is called a {\it rectifiable curve}.
It is a standard result that a connected set $\Gamma$ in a doubling metric space, which satisfies $H^1(\Gamma)<\infty$ is a rectifiable curve.  See for example the appendix of \cite{FFP}.

\section{Lemmas: Future balls}
Given a point $p \in \H$ and a horizontal line $L \subset \H$, if $\pi(p) \notin \pi(L)$, we let $p_L \in \H$ denote the point that is co-horizontal with $p$ ({\it i.e.} $NH(p_L^{-1}p) = 0$) so that $\pi(p_L) \in \pi(L)$ and the line in $\R^2$ spanned by $\pi(p)$ and $\pi(p_L)$ is perpendicular to $\pi(L)$.  If $\pi(p) \in \pi(L)$, then $p_L = p$.  It is easy to see that $d(p_L,L)$ scales like the square root of the $z$-distance from $p_L$ to $L$.  We have the following lemma.

\begin{lemma} \label{l:shortest-to-line}
  \begin{align*}
    \frac{1}{2} \left( d(p,p_L)^4 + d(p_L,L)^4 \right)^{1/4} \leq d(p,L) \leq 2 \left( d(p,p_L)^4 + d(p_L,L)^4 \right)^{1/4}.
  \end{align*}
\end{lemma}

\begin{proof}
  The inequality on the right hand side is simply the triangle inequality along with Jensen's inequality.

  By a rotation and translation, we may suppose that $L$ is the $x$-axis and the $x$ coordinate of $p$ is 0.  Thus, $p = (0,y,z)$ and $p_L = (0,0,z)$.  We have that
  \begin{align*}
    d(p_L,L) = \inf_{t \in \R} (t^4 + z^2)^{1/4} = |z|^{1/2}.
  \end{align*}
  As $d(p,p_L) = |y|$, we see that we get the left hand inequality if we show for all $t \in \R$ that
  \begin{align*}
    f(t) := d((t,0,0),(0,y,z))^4 = (t^2 + y^2)^2 + (z - 2ty)^2 \geq \frac{1}{16} \left( y^4 + z^2 \right).
  \end{align*}
  Note that $f(t) \geq y^4$ always.

  We also have that
  \begin{align*}
    f(t) \geq ( z - 2ty )^2
  \end{align*}
  and so $f(t) \geq \frac{1}{4} z^2$ unless $t \in \left( \frac{z}{4y}, \frac{3z}{4y} \right)$.  But if $t$ is in this regime, then
  \begin{align*}
    f(t) \geq (y^2 + t^2)^2 \geq 2t^2y^2 \geq \frac{1}{8} z^2.
  \end{align*}
  Thus, $f(t) \geq \frac{1}{8}z^2$ always and so
  \begin{align*}
    f(t) \geq \frac{1}{2} \left(y^4 + \frac{1}{8}z^2\right) \geq \frac{1}{16} \left( d(p,p_L)^4 + d(p_L,L)^4 \right).
  \end{align*}
\end{proof}

We have made no effort to optimize constants in this lemma.  It will not be necessary to use optimal constants in this paper.

Given a horizontal line $L$ and a point $p \in \H$, we let $P_L(p)$ denote the point in $L$ that is just a vertical translate of $p_L$.  Sometimes we will abuse notation and treat $P_L(p)$ as a point in $\R$ corresponding to the linear isometry of $L$ with $\R$.  It will be clear by context whether we mean $P_L(p)$ as a point in $\H$ or a point in $\R \cong L$.

Note that if $\pi(p) \in \pi(L)$, then $d(p,L) = d(p,P_L(p))$.  Indeed, we may assume $p = 0$ in which case $L$ is a horizontal line going through the $z$-axis.  The statement then follows as the metric balls of the  Koranyi norm are convex bodies that are symmetric about the $z$-axis.  Thus, we get from Lemma \ref{l:shortest-to-line} that
\begin{align}
  d(p,P_L(p)) \leq d(p,p_L) + d(p_L,L) \leq 2^{3/4} (d(p,p_L)^4 + d(p_L,L)^4)^{1/4} \leq 4 d(p,L). \label{e:closest-to-line}
\end{align}

Let $a,b \in \H$.  We let $\Sigma_{a,b}$ denote the algebraic area of the closed path in $\R^2$ that comes from the projection to $\R^2$ of any horizontal path in $\H$ connecting $a$ to $b$ (so in $\R^2$ it goes from $\pi(a)$ to $\pi(b)$) and then subsequently going back to $\pi(a)$ via a straight line.  Note by Heisenberg geometry that the vertical coordinate of $a^{-1}b$ is exactly $\Sigma_{a,b}$ and so we get that $\Sigma_{a,b}$ is in fact independent of the chosen horizontal path in $\H$.  Thus, we also see that
\begin{align}
  NH(a^{-1}b) = 2|\Sigma_{a,b}|^{1/2}. \label{NH-Sigma}
\end{align}

Given a horizontal line $L \subset \H$, we let $\Sigma^L_{a,b}$ denote the algebraic area of the following closed path in $\R^2$ (with the specified orientation) that we describe.  It first goes from $\pi(a_L)$ to $\pi(a)$ by a straight line.  Then it follows the projection to $\R^2$ of any horizontal path in $\H$ connecting $a$ to $b$.  It then goes from $\pi(b)$ to $\pi(b_L)$ via a straight line before finally going back to $\pi(a_L)$ via another straight line.  We easily see that $\Sigma^L_{a,b}=\Sigma_{a,b} + T$, where $T$ is the algebraic area of the trapezoid in $\R^2$ that starts at $\pi(a)$ and goes to $\pi(b)$, $\pi(b_L)$, $\pi(a_L)$, before finally going back to $\pi(a)$.  Thus, we see that $\Sigma_{a,b}^L$ is also independent of the chosen horizontal curve in $\H$.

Note that the path we constructed above for $\Sigma^L_{a,b}$ is the projection to $\R^2$ of a horizontal path in $\H$ that goes from $a_L$ to $b_L$ before going back to $\pi(a_L)$ via a straight line in $\R^2$.  Thus, this path is a valid path for computing $\Sigma_{a_L,b_L}$.  We then have that
\begin{align}
  NH(a_L^{-1}b_L) \overset{\eqref{NH-Sigma}}{=} 2|\Sigma_{a_L,b_L}|^{1/2} = 2|\Sigma^L_{a,b}|^{1/2} \label{e:z-change}
\end{align}
for any $a,b \in \H$ and horizontal line $L \subset \H$.  A very useful property of $\Sigma^L_{a,b}$ is that it is additive.  That is, for $a,b,c \in \H$, we have
\begin{align*}
  \Sigma^L_{a,c} = \Sigma^L_{a,b} + \Sigma^L_{b,c}.
\end{align*}

\begin{lemma} \label{l:trap-unaffine}
  If $a,b \in \H$ and $L \subset \H$ is a horizontal line, then
  \begin{align*}
    \max\{d(a,L), d(b,L)\} \geq \frac{1}{2} |\Sigma_{a,b}^L|^{1/2}.
  \end{align*}
\end{lemma}

\begin{proof}
  Suppose that $\max\{ d(a,L), d(b,L) \} \leq \epsilon$.  Then by Lemma \ref{l:shortest-to-line}, we have that $\max\{d(a_L,L), d(b_L,L)\} \leq 2 \epsilon$.  Then we can write $a_L = g y$ and $b_L = h z$ for some $g,h \in L$ and $y,z \in Z(\H)$ for which $\max\{d(y,0),d(z,0)\} \leq 2\epsilon$.  It follows that $\tilde{\pi}(a^{-1}b) = g^{-1}h$.  As $y,z$ commute with all elements of $\H$, we get from \eqref{e:z-change} that
  \begin{align*}
    |\Sigma_{a,b}^L|^{1/2} = \frac{1}{2}NH(a_L^{-1}b_L) = \frac{1}{2} d(y^{-1} g^{-1}h z, g^{-1}h) = d(g^{-1}h y^{-1}z, g^{-1}h) = \frac{1}{2} d(y,z) \leq 2 \epsilon.
  \end{align*}
\end{proof}

\begin{lemma} \label{l:NH-min-beta}
  For every $a,b \in \H$ and every horizontal line $L \subset \H$ we have
  \begin{align*}
    \max\{d(a,L), d(b,L)\} \geq \frac{1}{16} \frac{NH(a^{-1}b)^2}{d(a,b)}.
  \end{align*}
\end{lemma}
{
We make the observation that the right hand side above is independent of $L$.  Note that this says that sets of two points in $\H$ can have a nonnegative $\beta$ quantity for a ball containing them. This cannot happen  in the Euclidean case.
}
%\RS{Maybe we should briefly mention that the space spanned by lines emanating from a point is 2 dimensional and not 4 dimensional (reference?  am i getting the number "2" right?, or is it "3").  Later on we have that $\beta_{p_1,p_2}$ is positive, which is non-intuitive if you have a euclidean background, and this is the reason why}

\begin{proof}
  We now fix a horizontal line $L$.  By Lemma \ref{l:shortest-to-line}, we see that $d(a,L) \geq \frac{1}{2} d(a,a_L)$ and $d(b,L) \geq \frac{1}{2} d(b,b_L)$.  Thus, we are done unless $\|\pi(b) - \pi(b_L)\| = d(b,b_L) \leq \frac{1}{8} \frac{NH(a^{-1}b)^2}{d(a,b)}$ and $\|\pi(a) - \pi(a_L)\| = d(a,a_L) \leq \frac{1}{8} \frac{NH(a^{-1}b)^2}{d(a,b)}$.  Now consider the trapezoid $T$ in $\R^2$ defined by the points $\pi(a), \pi(b), \pi(a_L), \pi(b_L)$.  As $\pi : \H \to \R^2$ is 1-Lipschitz, this trapezoid has area at most $\frac{1}{8} NH(a^{-1}b)^2$.  But
  \begin{align*}
    |\Sigma_{a,b}^L| \geq |\Sigma_{a,b}| - |T| \overset{\eqref{NH-Sigma}}{\geq} \frac{1}{8} NH(a^{-1}b)^2.
  \end{align*}
  Thus, Lemma \ref{l:trap-unaffine} tells us that
  \begin{align*}
    \max\{d(a,L),d(b,L)\} \geq \frac{1}{16} NH(a^{-1}b) \geq \frac{1}{16} \frac{NH(a^{-1}b)^2}{d(a,b)}.
  \end{align*}
  In the last inequality, we used the fact that $NH(a^{-1}b) \leq d(a,b)$.
\end{proof}

The next lemma says that a well connected set that goes from the center to outside a ball and is close to a horizontal line $L$ must have large diameter when projected onto $L$.

\begin{lemma} \label{l:flat-exit}
  Let $\delta < \frac{1}{100}$.  Let $B \subset \H$ be a ball and $L \subset \H$ be a horizontal line.  Suppose $\Kset = \{p_i\}_{i=1}^N \subset \H$ is a set such that
  \begin{align}
    &p_1 = \cent(B), \notag \\
    &d(p_1,p_N) > \rad(B), \notag \\
    &d(p_i,p_{i+1}) < \delta \diam(B) \label{e:delta-connected} \\
    &\sup_{z \in \Kset \cap B} d(z,L) \leq \frac{1}{100} \diam(B). \label{e:1/100-close}
  \end{align}
  Then
  \begin{align*}
    \sup_{x,y \in \Kset \cap B} |P_L(x) - P_L(y)| > \frac{1}{4} \diam(B).
  \end{align*}
\end{lemma}

\begin{proof}
  Let $p_j$ be such that $d(p_1,p_j) \leq \rad(B)$ and $d(p_1,p_{j+1}) > \rad(B)$.  Then
  \begin{align*}
    d(p_1,p_j) \overset{\eqref{e:delta-connected}}{\geq} d(p_1,p_{j+1}) - d(p_j,p_{j+1}) > \frac{49}{100} \diam(B).
  \end{align*}
  By \eqref{e:closest-to-line}, we have that
  \begin{align*}
    d(p_1,P_L(p_1)) \overset{\eqref{e:closest-to-line} \wedge \eqref{e:1/100-close}}{<} \frac{4}{100} \diam(B), \qquad d(p_j,P_L(p_j)) \overset{\eqref{e:closest-to-line} \wedge \eqref{e:1/100-close}}{<} \frac{4}{100} \diam(B).
  \end{align*}
  Thus,
  \begin{align*}
    |P_L(p_1) - P_L(p_j)| = d(P_L(p_1),P_L(p_j)) \geq d(p_1,p_j) - \frac{8}{100} \diam(B) \geq \frac{41}{100} \diam(B).
  \end{align*}
\end{proof}

%The following lemma will be crucial for our angle improvement step. 
The following lemma will be crucial for the proof of Lemma \ref{l:angle-improvement}.  It says the following fact.  Let $E$ be a well connected and well spread out set and $L$ be a horizontal line.  If the distance of $\pi(E)$ to $\pi(L)$ in $\R^2$ is relatively large compared to the distance of $E$ to $L$ in $\H$, then either $\pi(E)$ must curve towards $\pi(L)$ or there exists some subball that has a large $\beta$.
\begin{lemma} \label{l:sharp-turn}
  Let $p < 4$, $\epsilon, M, M_1, \delta > 0$ such that
  \begin{align*}
\frac1{100}>    \epsilon > M_1 > \frac{M}{2} > 10 \delta > 0, \qquad \epsilon > M.
  \end{align*}
  Let $L \subset \H$ be a horizontal line.  Suppose $\{p_i\}_{i=1}^N \subset H$ is a sequence such that 
  \begin{align}
    &d(p_i,p_{i+1}) < \delta, \qquad \forall i \in \{1,...,N-1\} \notag \\
    &\max_{i \in \{1,...N\}} d(p_i,L) \leq \epsilon, \label{e:delta-eps-close} \\
    &\|\pi(p_1) - \pi(L)\| = \max_{i \in \{1,...,N\}} \|\pi(p_i) - \pi(L)\| = M_1, \label{e:M_1-close} \\
    &|P_L(p_1) - P_L(p_N)| > \frac{500\epsilon^2}{M} \notag.
  \end{align}
  Then either there exists $j \in \{1,...,N\}$ so that
  \begin{align*}
    &|P_L(p_j) - P_L(p_1)| < 500 \frac{\epsilon^2}{M}, \\
    &\|\pi(p_j) - \pi(L)\| < \frac{M_1}{2},
  \end{align*}
  or there exists a ball $B' \subset \H$ for which $\diam(B') \geq M$ and
  \begin{align*}
    \beta_\Kset(B')^p \diam(B') \geq 10^{-50} M.
  \end{align*}
\end{lemma}

\begin{proof}
  Suppose the first alternative is false, that is
  \begin{align}
    \|\pi(p_j) - \pi(L)\| \geq \frac{M_1}{2},  \label{e:pj-large-gap}
  \end{align}
  for all $j$ such that $|P_L(p_j) - P_L(p_1)| < 500 \frac{\epsilon^2}{M}$.  We will fix an order on $L$ so that $P_L(p_1) - P_L(p_N) > 0$.  We may suppose by removing a tail end of the sequence $\{p_i\}$ that $N$ is the first index for which $P_L(p_1) - P_L(p_N) > M \left\lceil \frac{400\epsilon^2}{M^2} \right\rceil$.  We then let $\Gamma$ denote the horizontal path connecting $p_1$ to $p_N$ that goes from $p_i$ to $p_{i+1}$ via a subriemannian geodesic.  As the  Koranyi metric and the Carnot-Carath\'{e}odory metric are biLipschitz equivalent and $\delta$ is small enough compared to $M$, we get that
  \begin{align*}
    \inf_{z \in \pi(\Gamma)} \|z - \pi(L)\| \overset{\eqref{e:cc-lower-bound} \wedge \eqref{e:pj-large-gap}}{>} \frac{M_1}{4}.
  \end{align*}
  We have by \eqref{e:delta-eps-close} and Lemma \ref{l:trap-unaffine} that
  \begin{align*}
    |\Sigma_{p_1,p_N}^L| < 16 \epsilon^2.
  \end{align*}
  %Consider the trapezoid $T$ in $\R^2$ defined by the points $\pi(p_1), \pi((p_1)_L), \pi(p_N), \pi((p_N)_L)$.  This trapezoid has unalgebraic area at least
  %\begin{align*}
  %  |T| \geq \frac{400\epsilon^2}{M} \frac{M}{4} \geq 100 \epsilon^2.
  %\end{align*}
  %We then have that
  %\begin{align*}
  %  |\Sigma_{p_1,p_N}| \geq |T| - |\Sigma_{p_1,p_N}^L| > 100\epsilon^2 - 16\epsilon^2 \geq 84 \epsilon^2.
  %\end{align*}
  Now let $N' = \lceil \frac{400\epsilon^2}{MM_1} \rceil$ and sequentially go through $\{p_i\}$ and choose a subsequence $\{q_i\}_{i=1}^{N'}$ such that $q_1 = p_1$, $q_{N'} = p_N$, and
  \begin{align*}
    P_L(q_j) - P_L(q_1) \in \left(jM_1 + \delta, jM_1 - \delta\right).
  \end{align*}
  This is possible because $\{p_i\}$ is a $\delta$-connected set.  Then we have that
  \begin{align*}
    \Sigma_{p_1,p_N}^L = \sum_{j=1}^{N'} \Sigma_{q_{j-1},q_j}^L,
  \end{align*}
  and so there exists some $j$ such that $|\Sigma_{q_{j-1},q_j}^L| \leq \frac{1}{N'} |\Sigma_{p_1,p_{N'}}^L| < \frac{16}{N'} \epsilon^2$.  Consider the trapezoid $T_i$ in $\R^2$ defined by the points $\pi(q_{j-1})$, $\pi((q_{j-1})_L)$, $\pi(q_j)$, and $\pi((q_j)_L)$.  We have that
  \begin{align*}
    |T_i| \geq (M_1 - 2\delta ) \frac{M_1}{4} > \frac{M_1^2}{5}.
  \end{align*}
  Here, we've used the fact that $\delta < \frac{M_1}{10}$.  Then
  \begin{align*}
    |\Sigma_{q_{j-1},q_j}| \geq |T_i| - |\Sigma_{q_{j-1},q_j}^L| \geq \frac{M_1^2}{5} - \frac{16 \epsilon^2}{N'} \geq \frac{M_1^2}{10}.
  \end{align*}

  Suppose first that $d(q_{j-1},q_j) < 10M_1$.  If we set $B' \subset \H$ to be the ball around $q_{j-1}$ of radius $10M_1$, then we get by Lemma \ref{l:NH-min-beta} that
  \begin{align*}
    \beta_{\{q_{j-1},q_j\}}(B')^p \diam(B') \geq \left( \frac{M_1^2/10}{25600 M_1^2} \right)^p 10M_1 \geq 10^{-25} M_1.
  \end{align*}
  Here, we've used the fact that $p < 4$.  As $M_1 > \frac{M}{2}$, we have found a ball $B'$ that satisfies the second alternative.

  Thus, we may suppose that $d(q_{j-1},q_j) \geq 10M_1$.  We have now two additional cases: either $\|\pi(q_{j-1}) - \pi(q_j)\| > 9M_1$ or $\|\pi(q_{j-1}) - \pi(q_j)\| \leq 9M_1$.  Consider the first subcase.  As $|P_L(q_{j-1}) - P_L(q_j)| \leq M_1 + 2\delta < 2M_1$, we get that $|P_{L^\perp}(\pi(q_{j-1})) - P_{L^\perp}(\pi(q_j))| > 4M_1$ where $L^\perp$ is a line in $\R^2$ that is perpendicular to $L$.  This means that
  \begin{align*}
    \max\{\|\pi(q_{j-1}) - \pi(L)\|,\|\pi(q_j) - \pi(L)\|\} \geq 2M_1,
  \end{align*}
  a contradiction of \eqref{e:M_1-close}.

  Thus, we are now in the subcase when
  \begin{align*}
    \|\pi(q_{j-1}) - \pi(q_j)\| \leq 9M_1 < \frac{9}{10} d(q_{j-1},q_j).
  \end{align*}
  As $d(q_{j-1},q_j)^4 = \|\pi(q_{j-1}) - \pi(q_j)\|^4 + NH(q_{j-1}^{-1}q_j)^4$, we get that
  \begin{align*}
    NH(q_{j-1}^{-1}q_j) \geq \frac{4}{5} d(q_{j-1},q_j).
  \end{align*}
  Thus, if we set $B' \subset \H$ to be the ball around $q_{j-1}$ of radius $2d(q_{j-1},q_j)$, we get from Lemma \ref{l:NH-min-beta} that
  \begin{align*}
    \beta_{\{q_{j-1},q_j\}}(B')^p \diam(B') \geq 25^{-p} 20M_1 \geq 10^{-50} M.
  \end{align*}
  This finishes the proof of the lemma.
\end{proof}

Given a ball $B \subset \H$, we let $\tilde{\beta}_\Kset(B)$ denote 
$\beta_{\pi(\Kset \cap B), \R^2}(\pi(B))$, 
the regular Jones-$\beta$-number \cite{Jones-TSP} of the projection of $\Kset \cap B$ to $\R^2$.  The following lemma is our angle improvement step, which says that there is either a subball $B'$ of large diameter with large $\tilde{\beta}_\Kset(B')$ ({\it i.e.} $\pi(\Kset)$ has a large angle) or there is some other subball $B''$ of large diameter with large $\beta_\Kset(B'')$.

\begin{lemma} \label{l:angle-improvement}
  Let $p < 4$, $\epsilon, M, \delta > 0$, and $\D > 1$ so that
  \begin{align}
    \epsilon &> M > 10^{10} \epsilon^2, \label{e:eps-M-ineq} \\
    M &> 100 \delta. \label{e:M-delta-ineq}
  \end{align}
  Let $B \subset \H$ be a ball and $L \subset \H$ be a horizontal line.  Suppose $\Kset \subset \H$ is a set such that $\Kset \cap B$ is $\delta \diam(B)$-connected in $\Kset \cap \D B$ and satisfies the following conditions:
  \begin{align}
    &\sup_{x,y \in \Kset \cap B}|P_L(x) - P_L(y)| \geq \frac{1}{4} \diam(B), \label{e:L-diam}\\
    &\sup_{z \in \Kset \cap \D B} d(z,L) \leq \epsilon \diam(B), \notag \\
    &\sup_{z \in \Kset \cap B} \|\pi(z) - \pi(L)\| = M \diam(B). \notag
  \end{align}
  Then either there exists a subball $B' \subseteq 2\D B$ whose center is a point in $\Kset$ for which if $L'$ is a horizontal line that realizes $\beta_\Kset(B')$ 
  such that
  %then
  \begin{align*}
    &\diam(B') = 30000 \frac{\epsilon^2}{M} \diam(B), \\
    &\tilde{\beta}_\Kset(B') \geq 10^{-10} \frac{M^2}{\epsilon^2}, \\
    &\sup_{x,y \in \Kset \cap B'} |P_{L'}(x) - P_{L'}(y)| \geq \frac{1}{4} \diam(B'),
  \end{align*}
  or there exists some other subball $B'' \subseteq 2\D B$ for which
  \begin{align*}
    &\diam(B'') \geq M \diam(B), \\
    &\beta_\Kset(B'')^p \diam(B'') \geq 10^{-50} M \diam(B).
  \end{align*}
\end{lemma}

\begin{proof}
  We may suppose that $\diam(B) = 1$ by dilation.  Let $a \in \Kset$ be a point such that $\|\pi(a) - \pi(L)\| > \frac{M}{2}$.  By \eqref{e:L-diam}, there exists some point $b \in \Kset$ so that $|P_L(a) - P_L(b)| \geq \frac{1}{8}$.  Let $\{p_j\}_{j=1}^N$ be a sequence in $\Kset \cap \D B$ such that $p_1 = a$, $p_N = b$, and $d(p_j,p_{j+1}) < \delta$ for all $j \in \{1,...,N-1\}$.

  We choose an index $i \in \{1,...,N\}$ such that
  \begin{align*}
    \|\pi(p_i) - \pi(L)\| = \sup_{j \in \{1,...,N\}} \|\pi(p_j) - \pi(L)\| =: M_1 > \frac{M}{2}.
  \end{align*}
  Note that we still have $M_1 \leq \epsilon$.

  If $|P_L(p_i) - P_L(p_N)| \geq |P_L(p_i) - P_L(p_1)|$, then let $\{q_j\}_{j=1}^{N'}$ denote the sequence $\{p_{i+j-1}\}_{j=1}^{i+N+1}$; otherwise, let $\{q_j\}_{j=1}^{N'}$ denote the sequence $\{p_{i-j+1}\}_{j=1}^i$.  That is, $\{q_j\}$ is the subsequence of $\{p_j\}$ that starts from $p_i$ and goes to $p_1$ or $p_N$, whichever is further along $L$.  %For convenience, we will suppose that $\|\pi(p_i) - \pi(L)\| = M$.  As we will only look at the points $\{q_j\}$ from now on, this will not change our results by more than some small constant factor.

  By truncating a tail end of $\{q_j\}$, we may now suppose that
  \begin{align*}
    &5000 \frac{\epsilon^2}{M} \leq |P_L(q_1) - P_L(q_{N'})| < 5000 \frac{\epsilon^2}{M} + \delta, \\
    &|P_L(q_1) - P_L(q_j)| < 5000 \frac{\epsilon^2}{M}, \qquad \forall i \in \{1,...,N'-1\}.
  \end{align*}
  
  Suppose first that there exists some $j \in \{1,...,N'\}$ so that
  \begin{align*}
    d(q_1,q_j) \geq 25000 \frac{\epsilon^2}{M}.
  \end{align*}
  As $d(q_1,q_j)^4 = \|\pi(q_1) - \pi(q_j)\|^4 + NH(q_1^{-1}q_j)^4$, we get that
  \begin{align*}
    NH(q_1^{-1}q_j) \geq \frac{1}{2} d(q_1,q_j).
  \end{align*}
  Then if we set $B''$ to be a ball around $q_1$ of radius $2d(q_1,q_j) \geq 50000 \frac{\epsilon^2}{M} \geq 50000 M$, Lemma \ref{l:NH-min-beta} gives that
  \begin{align*}
    \beta_{\{q_1,q_j\}}(B'')^p \diam(B'') \geq 128^{-p} 4 d(q_1,q_j) \geq 10^{-10} \frac{\epsilon^2}{M} \overset{\eqref{e:eps-M-ineq}}{\geq} 10^{-10} M.
  \end{align*}
  This $B''$ would be give the needed $B''$ to finish the proof of the lemma.
  
  Thus, we may suppose that
  \begin{align*}
    d(q_1,q_j) < 25000 \frac{\epsilon^2}{M}, \qquad \forall i \in \{1,...,N\}.
  \end{align*}
  Then by applying Lemma \ref{l:sharp-turn} we get that either there exists an $i \in \{2,...,N'\}$ such that
  \begin{align*}
    |P_L(q_i) - P_L(q_1)| < \frac{500\epsilon^2}{M}
  \end{align*}
  and
  \begin{align*}
    \|\pi(q_i) - \pi(L)\| < \frac{M_1}{2},
  \end{align*}
  or there exists a subball $B'' \subseteq B$ of diameter at least $M$ for which
  \begin{align*}
    \beta_\Kset(B'')^p \diam(B'') \geq 10^{-50} M.
  \end{align*}
  We may assume the first alternative as the second alternative would give the needed $B''$ to finish the proof of the lemma.

  Collecting everything together, we now have three points $q_1,q_i,q_{N'}$ so that
  \begin{align}
    &\max\{d(q_1,q_i),d(q_1,q_{N'})\} < 25000 \frac{\epsilon^2}{M}, \notag \\
    &\|\pi(q_N) - \pi(L)\| \leq \|\pi(q_1) - \pi(L)\| = M_1 \label{e:q1-qN'-loose}, \\
    &\|\pi(q_i) - \pi(L)\| < \frac{M_1}{2}, \label{e:qi-tight} \\
    &|P_L(q_i) - P_L(q_1)| < 500 \frac{\epsilon^2}{M}, \label{e:q1-qi-close} \\
    &5000 \frac{\epsilon^2}{M} \leq |P_L(q_1) - P_L(q_N)|, \label{e:q1-qN'-far} \\
    &M_1 \geq \frac{M}{2}. \notag
  \end{align}
  It is then elementary, although tedious, to show that if $L'$ is a line in $\R^2$ such that
  \begin{align*}
    \max\{ \|\pi(q_N) - L'\|, \|\pi(q_1) - L'\| \} \leq \frac{M'}{100},
  \end{align*}
  then $\|\pi(q_i) - L'\| \geq \frac{M'}{100}$.  This is because if a line $L'$ stays too close to $\pi(q_1)$ and $\pi(q_N)$, then as \eqref{e:qi-tight} is true, the slope of $L'$ must be too shallow to get close to $\pi(q_i)$.  Details are left to the reader.  Thus, if we let $B'$ be a ball around $q_1$ of radius $30000 \frac{\epsilon^2}{M}$, then
  \begin{align*}
    \tilde{\beta}_\Kset(B') \geq \frac{M'/100}{30000 \epsilon^2/M} \geq 10^{-10} \frac{M^2}{\epsilon^2}.
  \end{align*}
  
  Now suppose $\beta_\Kset(B') \geq \frac{1}{100}$.  Then
  \begin{align*}
    \beta_\Kset(B')^p \diam(B') = 30000 \beta_\Kset(B')^p \frac{\epsilon^2}{M} \diam(B) \geq 10^{-4} \frac{\epsilon^2}{M} \diam(B) \overset{\eqref{e:eps-M-ineq}}{>} 10^{-4} M \diam(B).
  \end{align*}
  Notice also that $\diam(B') = 30000 \frac{\epsilon^2}{M} \diam(B) > M \diam(B)$.  We would then get the needed $B''$ to finish the proof of the lemma if we set $B'' = B'$.

  Thus, we may suppose $\beta_\Kset(B') < \frac{1}{100}$.  Note that
  \begin{align*}
    \diam(B') = 30000 \frac{\epsilon^2}{M} \diam(B) \overset{\eqref{e:eps-M-ineq}}{<} \frac{1}{16} \diam(B).
  \end{align*}
  As $q_1$ was on a $\delta \diam(B)$-connected path from $a$ to $b$ for which $d(a,b) \geq \frac{1}{8} \diam(B)$, there exists a sequence $\{r_i\}_{i=1}^{N_1} \subset \Kset$ such that
  \begin{align*}
    &r_1 = q_1 = \cent(B'), \\
    &d(r_i,r_{i+1}) < \delta \diam(B), \\
    &d(r_1,r_{N_1}) > \frac{1}{2} \diam(B').
  \end{align*}
  We also have the estimate that
  \begin{align*}
    \delta \diam(B) = \delta \frac{\diam(B)}{\diam(B')} \diam(B') = \frac{1}{30000} \delta \frac{M}{\epsilon^2} \diam(B') \overset{\eqref{e:eps-M-ineq} \wedge \eqref{e:M-delta-ineq}}{<} \frac{1}{100} \diam(B').
  \end{align*}
  Thus, Lemma \ref{l:flat-exit} tells us that if $L'$ is the horizontal line that realizes $\beta_\Kset(B')$, then
  \begin{align*}
    \sup_{x,y \in \Kset \cap B'} |P_{L'}(x) - P_{L'}(y)| \geq \frac{1}{4} \diam(B').
  \end{align*}
  We then get the needed $B'$ to finish the proof of the lemma.
\end{proof}

The next lemma tells us that if the triangle inequality excess of three spread out points in a ball $B$ is large relative to $\beta_\Kset(B)$, then $\tilde{\beta}_\Kset(B)$ must also be large.  The $\D_2$ term is needed in its application.
%\RS{say something about the fact we abuse notation and write
{
Below we abuse notation and allow ourselves to write  $\R$ for 
$$\{(x,0,0)\in \H: x \textrm{ a is real number}\}$$
}
\begin{lemma} \label{l:large-R2-beta}
  Let $p < 4$, $0 < \alpha_1 < \alpha_2 < 1$, $\D > 0$, and $\D_2 > 1$.  Then there exist $\D_0 = \D_0(\alpha_1,\alpha_2) > 0$ and $\epsilon_0 = \epsilon_0(\alpha_1,\alpha_2,p,\D) \in (0,1)$ so that the following property holds.  Let $B \subset \H$ a ball and $\Kset \subseteq \H$ be a subset so that
  \begin{align*}
    \sup_{z \in \Kset \cap \D_2 B} d(z,\R) = \epsilon \diam(B),
  \end{align*}
  for some $\epsilon < \epsilon_0$.  If $p_1,p_2,p_3 \in \Kset \cap B$ so that $\alpha_1 \diam(B) \leq d(p_i,p_j) \leq \alpha_2 \diam(B)$ and
  \begin{align}
    d(p_1,p_2) + d(p_2,p_3) - d(p_1,p_3) = \eta \diam(B) \geq \D \epsilon^p \diam(B), \label{e:curvature-ineq-assumption}
  \end{align}
  then one of the $y$ coordinates of $p_i$ has absolute value at least $\frac{1}{\D_0} \eta^{1/2} \diam(B)$.
\end{lemma}

\begin{proof}
  Let $\D_3$ denote the minimal number such that
  \begin{align*}
    ((x + y)^2 + z)^{1/4} \leq x^{1/2} + \D_3(y + z),
  \end{align*}
  when $x,y,z \in \R^+$ satisfy the bounds $\frac{\alpha_1}{2} \leq x \leq \alpha_2$, $0 \leq y \leq 1$, and $0 \leq z \leq 1$.  That such a $\D_3$ exists follows from repeated use of Taylor's approximation and clearly depends only on $\alpha_1$ and $\alpha_2$.  Then we set $\D_0 = \max\{150, 150\D_3^{-1/2}\}$.

  Suppose the lemma is false, that is, we have \eqref{e:curvature-ineq-assumption} but the $y$ coordinates for all the $p_i$ have absolute value less than $\frac{1}{\D_0} \eta^{1/2} \diam(B)$.  We can dilate the setting so that $\diam(B) = 1$ and translate so that the $x$ coordinate of $p_2$ is 0.  We label the points $p_i = (x_i,y_i,z_i)$ so that $x_2 = 0$.  Then we have that
  \begin{multline*}
    d(p_1,p_2) + d(p_2,p_3) - d(p_1,p_3) \leq \left( (x_1^2 + (y_1 - y_2)^2)^2 + ( z_1 - z_2 + 2x_1y_2 )^2 \right)^{1/4} \\
    + \left( (x_3^2 + (y_2 - y_3)^2)^2 + ( z_2 - z_3 - 2x_2y_3 )^2 \right)^{1/4} - |x_1 - x_3|.
  \end{multline*}
  As we are supposing that $|y_i| < \frac{1}{\D_0} \eta^{1/2}$, we must have that $|y_1 - y_2| < \frac{2}{\D_0} \eta^{1/2}$.

  We also claim that
  \begin{align*}
    |\Sigma_{p_1,p_2}| = \frac{1}{4} \left| z_1 - z_2 + 2 x_1y_2 \right| < \frac{2}{\D_0} \eta^{1/2}.
  \end{align*}
  If not, as the $y$-coordinate of $\pi(p_1)$ and $\pi(p_2)$ are both less than $\frac{1}{\D_0} \eta^{1/2}$ and the $x$-coordinates differ by no more than 1, we get that the algebraic area of the trapezoid $T$ with corners $\pi(p_1)$, $\pi((p_1)_L)$, $\pi(p_2)$, and $\pi((p_2)_L)$ is no more than $\frac{1}{\D_0} \eta^{1/2}$.  Thus, we would have that
  \begin{align*}
    \left|\Sigma_{p_1,p_2}^L\right| \geq |\Sigma_{p_1,p_2}| - T > \frac{1}{\D_0} \eta^{1/2} \overset{\eqref{e:curvature-ineq-assumption}}{\geq} \frac{\D^{1/2}}{\D_0} \epsilon^{p/2}.
  \end{align*}
  As Lemma \ref{l:trap-unaffine} then says that
  \begin{align*}
    \max\{d(p_1,L),d(p_2,L)\} > \frac{\D^{1/4}}{2\D_0^{1/2}} \epsilon^{p/4}.
  \end{align*}
%  As $\epsilon < 4$, 
 % \RS{should be:
 {
  As $p < 4$, 
}
  we see that we would contradict the fact that $\beta_\Kset(B) = \epsilon$ if $\epsilon \leq \epsilon_0$ for some $\epsilon_0$ that we can set to depend only on $\D_0$, $\D$, and $p$.

  Finally, as $\eta \leq 2$, we have that both $|y_1 - y_2| < \frac{2}{\D_0} \eta^{1/2} \leq 1$ and $| z_1 - z_2 + 2x_1y_2 | < \frac{8}{\D_0} \eta^{1/2} \leq 1$.  Thus, we have by definition of $\D_3$ that
  \begin{align*}
    \left( (x_1^2 + (y_1 - y_2)^2)^2 + ( z_1 - z_2 + 2x_1y_2 )^2 \right)^{1/4} \leq |x_1| + \D_3 \left( |y_1 - y_2|^2 + | z_1 - z_2 + 2x_1y_2 |^2 \right).
  \end{align*}
  The same thing holds with $d(p_2,p_3)$ and so we get by our choice of $\D_0$ that
  \begin{multline*}
    d(p_1,p_2) + d(p_2,p_3) - d(p_1,p_3) \leq \\
    \D_3\left( (y_1 - y_2)^2 + ( z_1 - z_2 + 2x_1y_2 )^2 + (y_2 - y_3)^2 + ( z_2 - z_3 - 2x_2y_3 )^2 \right) < \eta,
  \end{multline*}
  a contradiction.
\end{proof}

\section{Main proposition}

\begin{remark}\label{r:DFFP}
In Proposition \ref{p:future-balls} below, we always have $q\in [2,p)$ or 
$$    d(p_1,p_2) + d(p_2,p_3) - d(p_1,p_3) \leq \D_{\FFP}\beta_\H(B)^2\diam(B)$$
The existence of the constant $\D_{\FFP}<\infty$  follows from Theorem 2.14 of \cite{FFP} (see equation (2.51) there).
\end{remark}

\begin{proposition} \label{p:future-balls}
  Let $p<4$, $0 < \alpha_1 < \alpha_2 < 1$ and  $\D_7>1$ be given. 
  Let $D=D_{\FFP}(\alpha_1/\D_7,\alpha_2/\D_7) > 0$ be the constant from Remark \ref{r:DFFP}.  
  There exists constants $\D_1 = \D_1(\alpha_1,\alpha_2,p, \D_7) > 0$ and $\epsilon_1(\alpha_1,\alpha_2,p, \D_7) \in (0,1)$ so that the following holds.  Let $B \subset \H$ be a ball and suppose $\Kset \subseteq \H$ is such that 
  \begin{align*}
    \beta_\Kset(\D_7 B) = \frac{\epsilon}{\D_7} \leq \epsilon_1(\alpha_1,\alpha_2,p).
  \end{align*}
  If $p_1,p_2,p_3 \in \Kset \cap B$ so that $\alpha_1 \diam(B) \leq d(p_i,p_j) \leq \alpha_2 \diam(B)$,
  \begin{align}\label{e:assumption-1-in-prop-1.8}
    d(p_1,p_2) + d(p_2,p_3) - d(p_1,p_3) = \D \epsilon^q \diam(B) > \D \epsilon^p \diam(B),
  \end{align}
  for some $q < p$, and for every subball $B' \subseteq 4D_7 B$ of diameter at least $\frac{1}{\D_1} \epsilon^{q/2} \diam(B)$, $\Kset \cap B'$ is $\frac{1}{\D_1} \epsilon^{q/2} \diam(B)$-connected
 inside $E\cap \D_7 B'$, then there exists a subball $B'' \subseteq 16\D_7 B$ of diameter
  \begin{align*}
    \diam(B'') \geq \frac{1}{\D_1} \epsilon^{q/2} \diam(B)
  \end{align*}
  so that
  \begin{align}\label{e:at-2}
    d(p_1,p_2) + d(p_2,p_3) - d(p_1,p_3) \leq \D_1\beta_\Kset(\D_7 B'')^p \diam(\D_7 B'').
  \end{align}
  and
  \begin{align}\label{e:at-4}
    \beta_\Kset(\D_7 B'')^p \leq \D_1\epsilon^{q/2} 
  \end{align}

\end{proposition}

\begin{proof}
  %By our choice of $\D$, FFP says $q \geq 2$.  
  We first choose $\epsilon_1$ small enough so that $\epsilon_1 \leq \epsilon_0(\alpha_1,\alpha_2,p,\D)$ where $\epsilon_0$ is from Lemma \ref{l:large-R2-beta}.  By rotation, we may assume that the horizontal line realizing $\beta_\Kset(\D_7 B)$ projects to the $x$-axis.  Then as $\beta_\Kset(\D_7 B) = \frac{\epsilon}{\D_7}$, Lemma \ref{l:large-R2-beta} says there exists a constant $\D_0$ so that
  \begin{align}
    M := \frac{1}{\diam(B)} \sup_{z \in \Kset \cap B} \| \pi(z) - \pi(L) \| \geq \frac{\D^{1/2}}{\D_0} \epsilon^{q/2}. \label{e:M-epsq-ineq}
  \end{align}
  As $2 \leq q < 4$, if we set $\epsilon_1$ smaller than some constant depending only on $\D$ and $\D_0$, we then get that $\epsilon > M > 10^{10} \epsilon^2$.  Thus, an application of Lemma \ref{l:angle-improvement} gives us either a ball $B' \subset 2 \D_7 B$ for which
  \begin{align}
    &\diam(B') = 30000 \frac{\epsilon^2}{M} \diam(B), \notag \\
    &\tilde{\beta}_\Kset(B') \geq 10^{-10} \frac{M^2}{\epsilon^2}, \label{e:tbeta1-ineq}
  \end{align}
  or some other ball $B'' \subset 2 \D_7 B$
  \begin{align*}
    &\diam(B'') \geq M \diam(B), \\
    &\beta_\Kset(B'')^p \diam(B'') \geq 10^{-50} M \diam(B).
  \end{align*}
  If we have the latter case, then as $M \geq \frac{\D^{1/2}}{\D_0} \epsilon^{q/2} \geq \frac{\D^{1/2}}{\D_0} \epsilon^q$, we get that $B''$ is our needed ball if we specify $\D_1$ large enough.  Thus, we may suppose that we have a ball that satisfies the conditions in the first case.  Let us denote this ball $B_1$.

  We let $L_1$ denote the horizontal line that realizes the infimum of $\beta_\Kset(\D_7 B_1)$.  Then
  \begin{align}
    M_1 := \frac{1}{\diam(B_1)} \sup_{z \in \Kset \cap B_1} \|\pi(z) - \pi(L_1)\| \geq \tilde{\beta}_\Kset(B_1). \label{e:M1-tbeta-ineq}
  \end{align}
  We then let $\alpha_1 \in [0,1]$ be such that
  \begin{align}
    \beta_\Kset(\D_7 B_1) = \frac{M_1^{\alpha_1}}{\D_7}.
  \end{align}
  Suppose $\alpha_1 \leq \frac{2}{p}$.  Then as $\tilde{\beta}_\Kset(B_1) \leq 1$, we have that
  \begin{multline*}
    \beta_\Kset(\D_7 B_1)^p \diam(\D_7 B_1) = \frac{M_1^{p\alpha_1}}{\D_7^{p-1}} \diam(B_1) \overset{\eqref{e:M1-tbeta-ineq}}{\geq} \frac{1}{\D_7^{p-1}} \tilde{\beta}_\Kset(B_1)^2 \diam(B_1) \geq \D_7^{1-p} 10^{-20} \frac{M^3}{\epsilon^2} \diam(B) \\
    \overset{\eqref{e:M-epsq-ineq}}{\geq} \frac{\D^{3/2}}{10^{20} \D_7^{p-1} \D_0^3} \epsilon^{\frac{3q}{2} - 2} \diam(B) \geq \frac{\D^{3/2}}{10^{20} \D_7^{p-1} \D_0^3} \epsilon^q \diam(B).
  \end{multline*}
  In the last inequality, we used the fact that $q < p < 4$.  This would give that $\D_7 B_1$ is a ball that would satisfy the claim of the proposition for sufficiently large $\D_1$.  Thus, we may suppose that $\alpha_1 > \frac{2}{p}$.

  Now suppose $M_1 \leq 10^{10} M_1^{2\alpha_1}$, that is,
  \begin{align*}
    M_1 \geq 10^{-10 (1-2\alpha_1)}.
  \end{align*}
  As $p$ is some fixed number strictly less than 4 and $\alpha_1 > \frac{2}{p}$, we get that there exists some $C > 0$ depending only on $p$ so that $M_1 > C$.  Thus,
  \begin{multline*}
    \beta_\Kset(\D_7 B_1)^p \diam(\D_7 B_1) = \frac{M_1^{\alpha_1 p}}{\D_7^{p-1}} 30000 \frac{\epsilon^2}{M} \diam(B) \geq 30000 \frac{C^{\alpha_1 p}}{\D_7^{p-1}} M \diam(B) \\
    > 30000 \frac{C^{\alpha_1 p}D^{1/2}}{\D_7^{p-1}\D_0} \epsilon^{q/2} \diam(B) \geq 30000 \frac{C^{\alpha_1 p}D^{1/2}}{\D_7^{p-1} \D_0} \epsilon^q \diam(B).
  \end{multline*}
  Again, we would have that $B_1$ is a ball that would satisfy the claim of the proposition for sufficiently large $\D_1$.  Thus, we may suppose $M_1 > 10^{10} M_1^{2\alpha_1}$.

  We now have the following information about $L_1$ and $B_1$:
  \begin{align*}
    &M_1^{\alpha_1} > M_1 > 10^{10} M_1^{2\alpha_1}, \\
    &\sup_{x,y \in \Kset \cap B_1} |P_{L_1}(x) - P_{L_1}(y)| \geq \frac{1}{4} \diam(B_1), \\
    &\sup_{z \in \Kset \cap \D_7 B_1} d(z,L_1) = M_1^{\alpha_1} \diam(B_1), \\
    &\sup_{z \in \Kset \cap B_1} \|\pi(z) - \pi(L_1)\| = M_1 \diam(B_1), \\
    &\diam(B_1) = 30000 \frac{\epsilon^2}{M} \diam(B).
  \end{align*}

  Suppose that we have a sequence of subballs $B_1,...,B_m$ with the following properties.  Each $B_j$ is contained in $2\D_7 B_{j-1}$.  If $L_k$ is the horizontal line realizing $\beta_\Kset(\D_7 B_k)$ then
  \begin{align*}
    \sup_{x,y \in \Kset \cap B_k} |P_{L_k}(x) - P_{L_k}(y)| \geq \frac{1}{4} \diam(B_k).
  \end{align*}
  If
  \begin{align*}
    M_k := \frac{1}{\diam(B_1)} \sup_{z \in \Kset \cap B_k} \|\pi(z) - \pi(L_k)\| \geq \tilde{\beta}_\Kset(B_k),
  \end{align*}
  then
  \begin{align*}
    \beta_\Kset(\D_7 B_k) = \frac{M_k^{\alpha_k}}{\D_7},
  \end{align*}
  for some $\alpha_k \in (2/p,1]$.  Furthermore we have the estimates
  \begin{align}
    &M_k^{\alpha_k} > M_k > 10^{10} M_k^{2\alpha_k}, \notag \\
    &M_k \geq 10^{-10k} \left( \frac{M^2}{\epsilon^2} \right)^{2^{k-1} (1-\alpha_1) \cdots (1-\alpha_{k-1})}, \label{e:Mk-ineq} \\
    &\diam(B_k) = 30000^k \left( \frac{M^2}{\epsilon^2} \right)^{1 - 2^{k-1} (1-\alpha_1) \cdots (1-\alpha_{k-1})} \frac{\epsilon^2}{M} \diam(B). \label{e:diam-Bk}
  \end{align}
  Here, the term $2^{k-1} (1-\alpha_1) \cdots (1-\alpha_{k-1})$ is understood to be 1 if $k =1$.  Then
  \begin{multline*}
    \frac{1}{\D_1} \epsilon^{q/2} \diam(B) = \frac{1}{\D_1} \epsilon^{q/2} \frac{\diam(B)}{\diam(B_m)} \diam(B_m) \\
    \overset{\eqref{e:diam-Bk}}{=} \frac{1}{30000^m \D_1} M\epsilon^{\frac{q}{2}-2} \left( \frac{\epsilon^2}{M^2} \right)^{1 - 2^{m-1} (1-\alpha_1) \cdots (1-\alpha_{m-1})}\diam(B_m) \leq \frac{1}{100} M_m \diam(B_m),
  \end{multline*}
  where in the last inequality we used the fact we can set $\D_1$ large enough and that
  \begin{align*}
    M_m \overset{\eqref{e:Mk-ineq}}{\geq} 10^{-10m} \left( \frac{\epsilon^2}{M^2} \right)^{1-2^{m-1} (1-\alpha_1) \cdots (1-\alpha_{m-1})} \frac{M^2}{\epsilon^2} \overset{\eqref{e:M-epsq-ineq}}{\geq} \frac{10^{-10m}D^{1/2}}{\D_0} \left( \frac{\epsilon^2}{M^2} \right)^{1-2^{m-1} (1-\alpha_1) \cdots (1-\alpha_{m-1})} M \epsilon^{\frac{q}{2}-2}.
  \end{align*}

  Thus, we can apply Lemma \ref{l:angle-improvement} to $B_m$ to give us either a ball $B' \subseteq 2\D_7 B_m$ so that
  \begin{align*}
    &\diam(B') = 30000 \frac{M_m^{2\alpha_m}}{M_m} \diam(B_m) \overset{\eqref{e:diam-Bk}}{=} 30000^{m+1} \left( \frac{M^2}{\epsilon^2} \right)^{1 - 2^m (1-\alpha_1) \cdots (1-\alpha_m)} \frac{\epsilon^2}{M} \diam(B), \\
    &\tilde{\beta}_\Kset(B') \geq 10^{-10} M_m^{2(1-\alpha_m)} \overset{\eqref{e:Mk-ineq}}{\geq} 10^{-10(m+1)} \left( \frac{M^2}{\epsilon^2} \right)^{2^m (1-\alpha_1) \cdots (1-\alpha_m)},
  \end{align*}
  or some other ball $B'' \subseteq 2\D_7 B_m$
  \begin{align*}
    &\diam(B'') \geq M_m \diam(B_m) \overset{\eqref{e:Mk-ineq} \wedge \eqref{e:diam-Bk}}\geq 10^{-10k} \left( \frac{M^2}{\epsilon^2} \right) \frac{\epsilon^2}{M} \diam(B) \geq 10^{-10k} M \diam(B), \\
    &\beta_\Kset(B'')^p \diam(B'') \geq 10^{-50} M_m \diam(B_m) \geq 10^{-10k-50} M \diam(B).
  \end{align*}
  If we have the latter case, then as $M \geq \frac{D^{1/2}}{\D_0} \epsilon^{q/2} \geq \frac{\D^{1/2}}{\D_0} \epsilon^q$, we get that $B''$ is our needed ball if we specify that $\D_1$ is large enough.  Thus, we can inductively construct these $B_k$.

  Like before, we let $L_{m+1}$ denote the horizontal line that realizes the infimum of $\beta_\Kset(\D_7 B_{m+1})$.  Then
  \begin{align*}
    M_{m+1} := \frac{1}{\diam(B_{m+1})} \sup_{z \in \Kset \cap B_{m+1}} \|\pi(z) - \pi(L_{m+1})\| \geq \tilde{\beta}_\Kset(B_{m+1}).
  \end{align*}
  We then let $\alpha_{m+1} \in [0,1]$ be such that
  \begin{align*}
    \beta_\Kset(\D_7 B_{m+1}) = \frac{M_{m+1}^{\alpha_{m+1}}}{\D_7}.
  \end{align*}
  As before, we may suppose that $\alpha_{m+1} > \frac{2}{p}$ and that $M_{m+1} > 10^{10} M_{m+1}^{2\alpha_{m+1}}$ as otherwise we would be done if we specify that $\D_1$ is large enough.  Thus, we have exhibited a subball $B_{m+1}$ that allows us to apply Lemma \ref{l:angle-improvement} again.

  Continuing inductively, we get that for each $k > 0$, if we specify $\D_1$ large enough, then we can find subballs satisfying \eqref{e:Mk-ineq} and \eqref{e:diam-Bk} (if such a ball does not exist, then sometime during the induction, we would have found a ball that satisfies the conclusion of the proposition).  Note that $\D_1$ for now depends on the $k$ that we specify.
  
  Let $m$ be the smallest integer such that $2^{m-1} \left( 1 - \frac{2}{p} \right)^{m-1} \leq \frac{1}{2}$.  Such a number exists as $p < 4$ and so $1 - \frac{2}{p} < \frac{1}{2}$.  We then see that $m$ is a constant depending only on $p$.  As $\alpha_m p > 2$, we get that
  \begin{multline*}
    \beta_\Kset(\D_7 B_m)^p \diam(\D_7 B_m) \geq \frac{M_m^2}{\D_7^{p-1}} \diam(B_m) \\
    \overset{\eqref{e:Mk-ineq} \wedge \eqref{e:diam-Bk}}{\geq} \frac{1}{10^{10m}\D_7^{p-1}} \left( \frac{M^2}{\epsilon^2} \right)^{2^{m-1} (1-\alpha_1) \cdots (1-\alpha_{m-1})} M \diam(B) = (*).
  \end{multline*}
  Note that
  \begin{align*}
    2^{m-1} (1 - \alpha_1) \cdots (1 - \alpha_{m-1}) \leq 2^{m-1} \left( 1 - \frac{2}{p} \right)^{m-1} \leq \frac{1}{2},
  \end{align*}
  where we again used the fact that $\alpha_k > \frac{2}{p}$.  As $M \leq \epsilon$, we get that
  \begin{align*}
    (*) \geq \frac{1}{10^{10m}\D_7^{p-1}} \frac{M^2}{\epsilon} \diam(B) \overset{\eqref{e:M-epsq-ineq}}{\geq} \frac{1}{10^{10m}\D_7^{p-1}} \frac{\D}{\D_0^2} \epsilon^{q-1} \diam(B) \geq \frac{1}{10^{10m}\D_7^{p-1}} \frac{\D}{\D_0^2} \epsilon^q \diam(B).
  \end{align*}
  As $m$ is some constant depending only on $p$, so is the needed $\D_1$ and so we get that by choosing $\D_1$ even larger
  \begin{align}
    d(p_1,p_2) + d(p_2,p_3) - d(p_1,p_3) \leq \D_1 \beta_\Kset(\D_7 B_m)^p \diam(\D_7 B_m). \label{e:at-3}
  \end{align}

  Note that $B_m$ satisfies all the properties of $B''$ except for possibly \eqref{e:at-4}.  In order to accomplish this, we iteratively double $B_m$ until we get a ball containing $4\D_7 B$ or until one more doubling will give us that \eqref{e:at-3} will be violated.  We note that in each doubling, the right hand side of \eqref{e:at-3} goes down by at most a factor of $2^{p-1}$. Call the resulting ball $B''$.
If we stopped because of the former condition, then
$$\beta_\Kset(\D_7 B'')\leq  4\beta_\Kset(16\D_7 B) \leq  \frac1{4 \D_7} \epsilon\,.$$
If we stopped because of the latter condition, then
$$D\epsilon^q \diam(B)\geq 
2^{1-p} D_1 \beta_\Kset(\D_7 B'')^p\diam(\D_7B'') 
\geq
2^{1-p} D_1 \beta_\Kset(\D_7 B'')^p\frac{1}{\D_1}\epsilon^{q/2}\diam(B)$$
which gives
$$D\epsilon^{q/2}
\geq
2^{1-p}  \beta_\Kset(\D_7 B'')^p\,.$$
Combining the two estimates gives \eqref{e:at-4} by making $\D_1$ large enough.
 
\end{proof}

%%%%%%%%%%%%%%%%%%%%%%%%%%%%%%%%%
%%%%%%%%%%%%%%%%%%%%%%%%%%%%%%%%%%
%%%%%%%%%%%%%%%%%%%%%%%%%%%%%%%%%%
%%%%%%%%%%%%%%%%%%%%%%%%%%%%%%%%%%
%%%%%%%%%%%%%%%%%%%%%%%%%%%%%%%%%%
%%%%%%%%%%%%%%%%%%%%%%%%%%%%%%%%%%
%%%%%%%%%%%%%%%%%%%%%%%%%%%%%%%%%%
%%%%%%%%%%%%%%%%%%%%%%%%%%%%%%%%%%
%12345678%%%%%%%%%%%%%%%%%%%%%%%%%%%%%%%%%%
%%%%%%%%%%%%%%%%%%%%%%%%%%%%%%%%%%
%%%%%%%%%%%%%%%%%%%%%%%%%%%%%%%%%%
%%%%%%%%%%%%%%%%%%%%%%%%%%%%%%%%%%
%%%%%%%%%%%%%%%%%%%%%%%%%%%%%%%%%%
%%%%%%%%%%%%%%%%%%%%%%%%%%%%%%%%%%
%%%%%%%%%%%%%%%%%%%%%%%%%%%%%%%%%%
%%%%%%%%%%%%%%%%%%%%%%%%%%%%%%%%%%
%%%%%%%%%%%%%%%%%%%%%%%%%%%%%%%%%%
%%%%%%%%%%%%%%%%%%%%%%%%%%%%%%%%%%

\section{The construction and its length}
Let $\Kset \subset \H$ be a set and $r \in (2,4)$.  We would like to construct $\Gamma\supset E$, a continuum, such that we control the length of $\Gamma$ by 
\begin{equation}\label{e:power-p-upper-bound}
\diam(\Kset) + \int_H \int_{t\geq 0} \beta_\H(B(x,r))^{\ourpower}  \frac{dt}{t^4}d\cH^4(x),
\end{equation}
which we are assuming is bounded.
To this end, we will use the algorithm from \cite{Jones-TSP}.
This had been done before in \cite{FFP}, where  the estimates established \eqref{e:power-p-upper-bound} for $\ourpower=2$.
We will follow the same notation and refer to the detailed work done in section 3 of \cite{FFP}.  In fact, with some appropriate choices, we will {\bf only need to modify the estimates} for one  of their cases in a non-trivial manner. 
In other words, the construction in \cite{FFP} works {\it as is} (but see Remark \ref{r:change-metric}!).  However, the estimate \eqref{e:power-p-upper-bound} for $\ourpower=2$ obtained in \cite{FFP} can be improved to yield
\eqref{e:power-p-upper-bound} for any $\ourpower\in(2,4)$, and that is what we do below.
We let 
$$p=\frac{\ourpower+4}2 < 4\,.$$

\begin{remark}\label{r:change-metric}
in \cite{FFP} the Carnot-Carath\'{e}odory metric was used for the definitions and construction.  We will use a different (equivalent)  metric, namely the Koranyi metric.  When we say we use the \cite{FFP} construction, we mean we use the same algorithm, but with respect to the Koranyi metric, i.e. all nets, balls etc. are with respect to this metric.
\end{remark}

\begin{remark}
 We will assume that the reader is very familiar with \cite{FFP} and has it on hand.  
In order to avoid confusion, in this section we use the exact same notation as in \cite{FFP}. 
We allow ourselves to reduce the value of the constant $\epsilon_0>0$ and increase the value of the constant $C_1>1$. We will assume in particular that $\epsilon_0<\epsilon_1$ of  Proposition \ref{p:future-balls}.
\end{remark}

\begin{remark}\label{r:d7}
We describe the dependency of constants below. All of them are allowed to depend on $\ourpower$.
When we will be invoking Proposition  \ref{p:future-balls},we will always do so with the constant
$$\D_7  = 2C_1\,.$$
$C_1$ needs to be large enough.  
%This is a constant that appears in \cite{FFP} and is required to be large there, and we use it for a similar purpose, however we need it to be even larger 
The constant $\D_1$ depends on $\ourpower, C_1$.
The constant $\D_{10}$ depends on $C_1$, $\D_7$ and $\ourpower$.
The constant $R$ depends on $C_1$, $\D_1$ and $\ourpower$.
The constant $\epsilon_0$ depends on $C_1$, $R$ and $\ourpower$.
The constant $C_{\ourpower}$ depends on $C_1$ and $\ourpower$.
The constant $\epsilon_0$  is the only one that needs to be sufficiently small, the rest  of the dependancies are lower bounds.
\end{remark}

The construction in \cite{FFP} is inductive, and there are 3 hypotheses which hold at every step of the process: (P1), (P2),(P3).  We will add two more (P4) and (P5),  and claim that they  hold as well.

\begin{itemize}
\item[(P4)]: For any $P\in \Delta_k$, we have that $\ball(P, C_12^{-k})\cap \Delta_k$ is connected via $\Gamma_k\cap \ball(P, C_12^{-k})$.  
%If $\beta(U_c(P, C_12^{-k}))<\epsilon_0$, then 
%$U_c(P, C_12^{-k})\cap \Delta_k$ is connected via 
%$\Gamma_k\cap U_c(P, C_12^{-k})\cap N$, where $N$ is a $2C_12^{-k}\beta(U_c(P, C_12^{-k}))$ of a horizontal line $L$. In addition, in this case we also have a that $U_c(P, C_12^{-k})\cap \Delta_k$  is connected in the order given by Proposition 3.2 of \cite{FFP}. 

\item[(P5)]: If $\beta_\H(\ball(P, C_12^{-k}))<\epsilon_0$ and $I\in \Gamma_k\cap \ball(P, C_12^{-k})$ is an interval, then $I$ is in the $2^{-k}C_1\epsilon_0$ neighborhood of $\Gamma_{k+1}$. Furthermore, there is a map $I\to I_1$ which take the interval $I$ to a polygonal curve $I_1\subset \Gamma_{k+1}$. This is the only way in which an interval $I$ may be deleted.
\end{itemize}

Suppose without loss of generality that $E$ is closed.
We now proceed precisely as in the construction of \cite{FFP}, however we will replace  the estimates of 
{\bf Case B1} with {\bf Case B1'}, specifically, we will improve on equation (3.4), p. 465 of \cite{FFP} via improving on the estimates for the quantities $\cS_1$, $\cS_2$. {\bf Case B2(i).2} will follow similar changes.
Let $B=\ball(P, C_1 2^{-j})$.

{\bf Case B1':} Since  $\beta_\H(B)<\epsilon_0$ (which we may, as we are not in Case A),  
there exists an order on $\Delta_j\cap \ball(P, C_12^{-j})=[P_1,...,P_n]$. We separate into two case, based on the validity of the assumptions for  Proposition \ref{p:future-balls}.

{\bf Case B1'(i):} For each triple of the form $[p_1,p_2,p_3]=...[P_{i_1}, P_{i_2}, P_{i_3}]$ where $1\leq i_1<i_2<i_3\leq n$ one of two things happens:  either we may apply  Proposition \ref{p:future-balls}, or assumption \eqref{e:assumption-1-in-prop-1.8}  of  Proposition \ref{p:future-balls} fails for all $q<p$.
%Specifically, [refer to the rest of the assumptions]. 

Suppose  Proposition \ref{p:future-balls} is applicable to $B$ and the triple of points $P_{i_1}, P_{i_2}, P_{i_3}$. Then it guaranties  a ball 
%$B''=B''[P_{i_1}, P_{i_2}, P_{i_3}]$. 
$B'''=B'''[P_{i_1}, P_{i_2}, P_{i_3}]$. 
There is a ball which we denote by  $F_1(B, i_1,i_2,i_3)$ with center $z\in \Delta_k$ and radius $C_12^{-k}$ for some $k\in \N$, such that $\frac{1}{1000}B'''\subset F_1(B, i_1,i_2,i_3) \subset 1000 B'''$.
By multiplying the  constant $D_1$ by a factor, we may assume without loss of generality that the conclusion of  Proposition \ref{p:future-balls} holds for $F_1(B, i_1,i_2,i_3)$.
Using this notation we proceed. As in \cite{FFP} and using Proposition \ref{p:future-balls} as well as the fact that the numbers of tuples $1\leq i_1<i_2<i_3\leq n$ is bounded independently of $B$ (or $E$) we have
$$\cS_1\leq 
	C\bigg( \beta_\H(B)^p\diam(B) + 
	\sum_{1\leq i_1<i_2<i_3\leq n \atop (*)} 
		\D_1\beta_\H(\D_7 F_1(B, i_1,i_2,i_3))^p
				\diam(\D_7 F_1(B, i_1,i_2,i_3))\bigg)\,,$$
where
%\begin{center}
%(*)    \quad    Proposition \ref{p:future-balls}  is applicable AND 
%$\beta_\H(B)<\beta_\H (\D_7 F_1(B, i_1,i_2,i_3))$.
%\end{center}
{
\begin{center}
(*)    \quad    Proposition \ref{p:future-balls}  is applicable.
\end{center}
}
Because of the first term, this inequality holds even if assumption \eqref{e:assumption-1-in-prop-1.8} fails for all $q<p$.
We have that when $F_1(B, i_1,i_2,i_3)$ is defined, it satisfies (for $q\in [2,4)$ depending on $[P_{i_1}, P_{i_2}, P_{i_3}]$)
\begin{equation*}%\label{e:rad-not-too-small-1}
\diam(F_1(B, i_1,i_2,i_3))\geq \frac1{\D_1}\beta_\H(B)^{q/2}\diam(B)\,. 
\end{equation*}
which may be weakened to 
\begin{equation*}%\label{e:rad-not-too-small-2}
\diam(F_1(B, i_1,i_2,i_3))\geq \frac1{\D_1}\beta_\H(B)^{p/2}\diam(B)\,. 
\end{equation*}

Similarly for $\cS_2$. This gives an improvement over equation (3.4) in \cite{FFP} (changing the value of $C$ to incorporate $\D_1$):
\begin{align*}
l(\Gamma_j)-l(\Gamma_{j-1})&\\ 
\leq&	C\bigg(\beta_\H(B)^p\diam(B) + 
	\sum_{1\leq i_1<i_2<i_3\leq n \atop (*)} 
		\beta_\H(\D_7 F_1(B, i_1,i_2,i_3))^p\diam(\D_7  F_1(B, i_1,i_2,i_3))\bigg)\,. 
\end{align*}
%\begin{remark}%\label{r:f-1-multiplicity}
% The function $B\to F_1(B,\cdot)$ is at most 
%$C\log(1/\beta_\H(B))\leq C\log(1/\beta_\H(\D_7 F_1(B,\cdot)))$ to 1.
%\RS{$C\log(1/\beta_\H(B))\leq C\log(1/\beta_\H(\D_7 F_1(B,\cdot)))$ is FALSE.  embarrassing.  I have an idea to fix this though.  Need to check before making another foolish statement.}
%Thus
%\begin{multline}%\label{e:sum-prop41applies}
%\sum\limits_{j\in\N}\sum\limits_{P\in \Delta_j} 
%\sum_{1\leq i_1<i_2<i_3\leq n \atop B=B(P,C_12^{-j}),\  (*)} 
%		\beta_\H(\D_7 F_1(B, i_1,i_2,i_3))^p\diam(\D_7  F_1(B, i_1,i_2,i_3)) \\
%\leq
%C_{\ourpower} \sum\limits_{j\in\N}\sum\limits_{P\in \Delta_j} 
%	\beta_\H(P,\D_{10} 2^{-j})^{\ourpower}2^{-j}
%\end{multline}
%where $C_{\ourpower}$ sufficiently large so that   $C x^p\log(1/x)<C_{\ourpower}x^{\ourpower}$ for $x\in [0,1]$.
%This will be used at the end of this section, when summing over all $B$.
%\end{remark}

%\begin{remark}\label{r:f-1-multiplicity}
% The function $B\to F_1(B,\cdot)$ is at most 
%$C\log(1/\beta_\H(B))$ to 1.
%\end{remark}
Let denote by $\cB_{B1'i}$  the collection of all balls which fall into Case B1'(i) and satisfy (*).
 Using \eqref{e:at-4}, the function $B\to F_1(B,\cdot)$ is at most 
$C\log(1/\beta_\H(B))$ to 1.
Thus
\begin{align*}
\sum\limits_{B\in \cB_{B1'i} } &
		\beta_\H(\D_7 F_1(B, i_1,i_2,i_3))^p\diam(\D_7  F_1(B, i_1,i_2,i_3)) \\
&\leq
\sum\limits_{t=0}^\infty
\sum\limits_{ B\in \cB_{B1'i}\atop  \beta_\H(B)\in[2^{-t}, 2^{-t+1})}
		\beta_\H(\D_7 F_1(B, i_1,i_2,i_3))^p\diam(\D_7  F_1(B, i_1,i_2,i_3)) \\
&\leq
\sum\limits_{t=0}^\infty 2^{-(t-1)(p-\ourpower)}
\sum\limits_{ B\in \cB_{B1'i}\atop   \beta_\H(B)\in[2^{-t}, 2^{-t+1})}
		\beta_\H(\D_7 F_1(B, i_1,i_2,i_3))^\ourpower \diam(\D_7  F_1(B, i_1,i_2,i_3)) \\
&\leq
\sum\limits_{t=0}^\infty t2^{-(t-1)(p-\ourpower)}
\sum\limits_{j\in\N}\sum\limits_{P\in \Delta_j} 
	\beta_\H(P,\D_{10} 2^{-j})^{\ourpower}2^{-j}\\
&\leq
C_{\ourpower}
\sum\limits_{j\in\N}\sum\limits_{P\in \Delta_j} 
	\beta_\H(P,\D_{10} 2^{-j})^{\ourpower}2^{-j}
\end{align*}
where $C_{\ourpower}=\sum\limits_{t=0}^\infty t2^{-(t-1)(p-\ourpower)}<\infty$ as $p=\frac{r+4}{2}$.
We summarize this as 
\begin{align}\label{e:sum-prop41applies}
\sum\limits_{B\in \cB_{B1'i} } 
		\beta_\H(\D_7 F_1(B, i_1,i_2,i_3))^p\diam(\D_7  F_1(B, i_1,i_2,i_3)) 
\leq
C_{\ourpower}
\sum\limits_{j\in\N}\sum\limits_{P\in \Delta_j} 
	\beta_\H(P,\D_{10} 2^{-j})^{\ourpower}2^{-j}\,.
\end{align}

{\bf Case B1'(ii):} In this case there is at least one (ordered) triple 
$[P_{i_1}, P_{i_2}, P_{i_3}]$,  $1\leq i_1<i_2<i_3\leq n$, where we cannot apply Proposition \ref{p:future-balls} and assumption \eqref{e:assumption-1-in-prop-1.8} holds for some $q\in[2,p)$. Let us fix such an instance $(i_1, i_2,i_3)$. 
Since we are not in Case A, we conclude the existence of a ball $F_2(B)=B'\subset 4\D_7 B$ of diameter 
$\diam(F_2(B))\geq \frac1{\D_1} \epsilon^{q/2}\diam(B)$ which has 
$E\cap F_2(B)$ is not $\frac1{\D_1} \epsilon^{q/2}\diam(B)$-connected 
inside $E\cap \D_7 F_2(B)$, 
where $\epsilon^q\diam(B)=\beta_\H(B)^q \diam(B)=\dist(P_{i_1}, P_{i_2}) + \dist (P_{i_2}, P_{i_3}) - \dist(P_{i_1},  P_{i_3})$. As stated in Remark \ref{r:d7}, $\D_7=2C_1$.

Let 
$$\alpha\in \left[ \frac1{\D_1} \epsilon^{q/2}\diam(B), \ 4\D_7\diam(B) \right]$$ 
be the largest number of the form $2^{-l}$ such that 
$E\cap F_2(B)$ is not $\alpha$-connected
inside $2C_1 F_2(B)$, 
but is $2\alpha$-connected 
inside $2C_1 F_2(B)$,
with  $F_2(B)\subset 4\D_7 B$ and  $\diam(F_2(B))\geq \alpha$. 
Let 
%We may assume $B'$ is the largest such ball, and 
 $x,y\in E\cap F_2(B)$ be two points of distance $\in (\alpha, 2\alpha]$ which witness to this (discrete) non-connectedness, and minimize $\dist(x,y)$.  
Let $k$ be such that $2^{-k}=\alpha/{128}$, and let 
$x',y'\in\Delta_k\cap \frac{11}{10}F_2(B)$ be minimize distance  to $x,y$ respectively. 
If $C_1>2^{10}$, then 
$ \ball(x',C_1 2^{-k})\supset \ball(x',3\alpha)\ni y'$, 
and thus (from  (P4)),  
the points $x',y'$ are connected via 
$\Gamma_k\cap \ball(x',C_12^{-k})$ 
with a polygon $P_{x,y}$ of edges $\geq 2^{-k}$.

Since segments are only modified in cases other than Case A, we have by (P5) that $P_{x,y}$ is in the 
$2\alpha C_1^2\epsilon_0$ neighborhood of the limit curve $\Gamma$, and furthermore, that there is an arc  $\Gamma_{x,y}\subset \Gamma$ which contains $P_{x,y}$ in its $2\alpha C_1^2\epsilon_0$ neighborhood. 
If we take $\epsilon_0$ small enough so that $C_1^2\epsilon_0<1/100$, then $x,y$ are in the $\alpha/10$ neighborhood of $\Gamma_{x,y}$.  
Furthermore,
$\Gamma_{x,y}\subset 2\ball(x',C_12^{-k}) \subset C_1 F_2(B)$.
Recall that $x,y$ are not $\alpha$-connected  inside $2C_1 F_2(B)$ and that 
$\rad(F_2(B))\geq \alpha/2$, which gives us the following lemma.
\begin{lemma}
There is  a 
connected set $\Gamma_B\subset \Gamma_{x,y}$ such that

%
%
%Thus, one of three options hold:  
%\begin{enumerate}
%\item[(a)]
%%$x$ and $y$ are not $\alpha$-connected in $3F_2(B)\cap E$,
%$x$ and $y$ are not $\frac{\alpha}3$-connected in $3F_2(B)\cap E$,
%\item[(b)]
%$\beta(10F_2(B))\geq \epsilon_0$
%\item[(c)]
%$\beta(10F_2(B))< \epsilon_0$ AND $x$ and $y$ are  $\frac{\alpha}3$-connected in $3F_2(B)\cap E$
%\end{enumerate}
%As an example, note that if  $3F_2(B)\subset 2B$, then option (a) must hold.
%
%If {\it option (a)} holds, then, as $\Gamma_{x,y}\subset 2F_2(B)$, we have 
% the existence of a connected set $\Gamma_B\subset \Gamma_{x,y}$ such that  

%at we may take a sub-arc $\gamma_B$ (of at least half the image length) such that

\begin{equation}\label{e:properties-Gamma-B-i}
 \Gamma_B\subset \Gamma
\end{equation}

\begin{equation}\label{e:properties-Gamma-B-ii}
 \frac{1}{10}\alpha \leq \diam(\Gamma_B) \leq  \frac2{10}\alpha
 \end{equation}

\begin{equation}\label{e:properties-Gamma-B-iii}
 \frac{1}{10}\alpha \leq \dist(\Gamma_B, \Kset) \leq   \alpha
\end{equation}
\end{lemma}
\begin{proof}
We proceed by contradiction. Suppose that any candidate connected set $G$ satisfying \eqref{e:properties-Gamma-B-i} and \eqref{e:properties-Gamma-B-ii} fails 
$\frac{1}{10}\alpha \leq \dist(G, \Kset) $.  
This means that $x'$ and $y'$ are $\frac4{10}\alpha$-connected in $E\cap (C_1+1) F_2(B)$, which is  contradiction since  $x,y$ are not $\alpha$-connected  inside $2C_1 F_2(B)$ and 
$\dist(x,x'), \dist(y,y')<\alpha/128$. 
Thus we have a connected $G$ satisfying \eqref{e:properties-Gamma-B-i} and \eqref{e:properties-Gamma-B-ii} and $\frac{1}{10}\alpha \leq \dist(G, \Kset) $. The connectedness of $\Gamma_{x,y}$ and the fact that $x'\in E$ then implies that there is also a $G$ satisfying the RHS of \eqref{e:properties-Gamma-B-iii} in addition to the above properties.
\end{proof}

Let  $\largeconst$ be a large constant to be chosen, and recall that $p=\frac{\ourpower+4}{2}$, i.e. 
 $\smallconst+ p/2 <2$. 
Suppose also that $\epsilon_0$ is small enough so that  $ \epsilon_0^{\smallconst+p/2}>\largeconst \epsilon_0^2 $.  Since we have 
 $q<p<4$ and $\beta_\H(B)<\epsilon_0$, this improves \eqref{e:properties-Gamma-B-ii} above to  
\begin{equation}\label{e:enough-room-to-spare}
H^1(\Gamma_{B}) \beta_\H(B)^\smallconst\geq  \largeconst \frac1{2\D_1} \beta_\H(B)^2 \diam(B) \,. 
\end{equation}
which means that the cost of the \cite{FFP} algorithm in this case is dominated by 
$H^1(\Gamma_{B}) \beta_\H(B)^\smallconst$.

\begin{lemma}\label{l:log}
Let  $x_0\in \Gamma$ and $t\in\N$.
Then
\begin{equation}\label{e:how-many-balls}
\#\{B=\ball(z,C_12^{-j}):j\in\N; \ z\in\Delta_j;\  \Gamma_B \ni x_0;\ \beta_\H(B)\in [2^{-t}, 2^{-t+1})\} <Ct.
\end{equation} 

%\begin{equation}\label{e:how-many-balls}
%\#\{B=\ball(z,C_12^{-j}):j\in\N; \ z\in\Delta_j;\  \Gamma_B \cap \Gamma_{B_1} \neq \emptyset;\ \beta_\H(B)\in [2^{-t}, 2^{-t+1})\} <Ct^2.
%\end{equation} 
\end{lemma}
Equation \eqref{e:how-many-balls} will be used in conjunction with \eqref{e:enough-room-to-spare} above later on. 
\begin{proof}
First note that 
since $\H$ is doubling, only a fixed number of balls of any fixed scale may intersect at a point.  This, together with the fact that
$$\Gamma_B\subset \Gamma_{x,y}\subset 2\D_7B$$
gives us a uniform bound for the number of balls on the left hand side of \eqref{e:how-many-balls} of a single scale.
We now address the question of how many scales can come into play.

The answer  will
follow from \eqref{e:properties-Gamma-B-ii}, \eqref{e:properties-Gamma-B-iii} above. 
Let $B_1=\ball(z_1,C_12^{-j_1})$ be a ball, and suppose $x_0\in \Gamma_{B_1}$.
Now we have that
 if $B$ is a ball on the left hand side of \eqref{e:how-many-balls}, then 
\begin{multline}\label{e:different-scales}
\diam(B)\leq 
10\D_12^{2t}\dist(\Gamma_B,E)
\leq
 10\D_12^{2t} 
 \big(\dist(\Gamma_B,\Gamma_{B_1}) + \diam(\Gamma_{B_1}) +  \dist(\Gamma_{B_1},E)  \big)
 \leq\\
 10\D_12^{2t} 
 \big(8\D_7 \diam(B_1)\big)
=C2^{2t}\diam(B_1)
\end{multline}
and in the same way,  
$\diam(B_1)\leq 
C2^{2t}\diam(B)
$.
Thus, only $O(t)$ of scales need to be considered, giving the lemma.

%Consider now a ball $B_2$ appearing in the left hand side of \eqref{e:how-many-balls} or radius 
%$\rad(B_2)=C_12^{-j_2}$.
%We claim that
%\begin{equation}
%\#\{B=\ball(z,C_12^{-j_2}):  z\in \Delta_{j_2};\ \dist(\Gamma_{B}, \Gamma_{B_2}) \leq \diam(\Gamma_{B_1}); \ \beta_\H(B)\in [2^{-t}, 2^{-t+1})\} < ........
%\end{equation}
%Indeed, if $z$ is the center of a ball $B$ on the left hand side of \eqref{e:how-many-balls}, and $z_2$ is the center of $B_2$, then
%\begin{multline}
%\dist(z,z_2)
%\leq \
%dist(z, \Gamma_B) +  \dist(\Gamma_{B}, \Gamma_{B_2}) + \dist(z_2, \Gamma_{B_2})
%\leq 
%\D_7 C_1 2^{-j_2} + \diam(\Gamma_{B_1}) + \D_7 C_1 2^{-j_2}\\
%\leq
%2\D_7 C_1 2^{-j_2} + \D_7C_12^{-j_1}
%\leq 
%C2^{2t}2^{-j_2}
%\end{multline}
%where the radius of $B_1$ is $C_12^{-j_1}$
%and 
%for the ultimate inequality we used \eqref{e:same-scale-bound}.

\end{proof}

Let denote by $\cB_{B1'ii}$  the collection of all balls which fall into Case B1'(ii).
We sum as follows, using \eqref{e:enough-room-to-spare} and Lemma \ref{l:log}.%\eqref{e:how-many-balls}.
\begin{multline}\label{e:sum-of-B1ii}
\sum\limits_{ B\in \cB_{B1'ii}}\beta_\H(B)^2 \diam(B) \leq
\sum\limits_{ B\in \cB_{B1'ii}} \frac{2\D_1}{\largeconst} H^1(\Gamma_{B}) \beta_\H(B)^\smallconst\\
\leq
\sum\limits_{t=0}^\infty
\sum\limits_{ B\in \cB_{B1'ii}\atop \beta_\H(B)\in[2^{-t}, 2^{-t+1})}
\frac{4\D_1}{\largeconst} H^1(\Gamma_{B}) 2^{-t\smallconst}
\leq
\frac{4\D_1}{\largeconst} 
 \sum\limits_{t=0}^\infty
 tH^1(\Gamma)2^{-t\smallconst}
=
 \frac{4\D_1}{\largeconst} 
 C(\ourpower) H^1(\Gamma)
\end{multline}
Finally, we note that increasing $\largeconst$ (which forces us to decrease  $\epsilon_0$ accordingly), reduces 
$ \frac{4\D_1}{\largeconst}  C(\ourpower) $ 
 to being  
 arbitrarily close to 0.
 
% 
% Let denote by $\cB_{B1'iib}$  the collection of all balls which fall into Case B1'(ii), option (b).
%We sum as follows, using \eqref{e:B1ii-option-b} and Remark \ref{r:how-many-scales}.
%\begin{align}
%&\sum\limits_{ B\in \cB_{B1'iib}}\beta(B)^2 \diam(B) \leq
%\frac{\D_1}{\epsilon_0^4} 
%\sum\limits_{ B\in \cB_{B1'iib}}\beta(B)^r \beta(F_2(B))^4 \diam(F_2(B))\\
%&\leq
%\sum\limits_{t=0}^\infty
%\sum\limits_{ B\in \cB_{B1'iib}\atop \beta(B)\in[2^{-t}, 2^{-t+1})}
%\frac{\D_1}{\epsilon_0^4} 
%2^{-tr}\beta(F_2(B))^4 \diam(F_2(B))\\
%&\leq
%\sum\limits_{t=0}^\infty
%\sum\limits_{ B'\in \allballs }
%\frac{\D_1}{\epsilon_0^4} 
%t2^{-tr}\beta(B')^4 \diam(B')\\
%&=
%C_r
%\sum\limits_{ B'\in \allballs }
%\frac{\D_1}{\epsilon_0^4} 
%\beta(B')^4 \diam(B')
%\end{align}

\smallskip

{\bf Case B2(i).2} appeals to the estimates in Case B1.  The estimates for Case B1' work to give Case B2(i).2'.

\smallskip

The rest of the cases follow with the same estimates as those in Section 3 of \cite{FFP}.
This allows us to improve the estimate at the top of page 468 of \cite{FFP} to give \eqref{e:power-p-upper-bound} (increasing their value of $C_1$, and allowing $C$ to change form line to line).  Sepcifically, we use equations \eqref{e:sum-prop41applies}, \eqref{e:sum-of-B1ii} to get:

\begin{align*}
H^1(\Gamma)
&\leq 
C\diam(\Kset) 
+ 
C\sum\limits_{j\in\N}\sum\limits_{P\in \Delta_j} \beta_\H(P,C_1  2^{-j})^p 2^{-j}
+ 
\frac{1}{10^6}H^1(\Gamma) +  \frac{1}{10^3}H^1(\Gamma)\\
&+  
C\sum\limits_{j\in\N}\sum\limits_{P\in \Delta_j} \beta_\H(\D_7F_1(B(P,C_1 2^{-j})))^p 2^{-j}
+ 
\frac{4D_1}{\largeconst}  C_{\ourpower} H^1(\Gamma)\\
\leq& 
C\diam(\Kset) 
+ 
C_{\ourpower}\sum\limits_{j\in\N}\sum\limits_{P\in \Delta_j} \beta_\H(P,\D_{10} 2^{-j})^{\ourpower}2^{-j}
+
\frac{1}{10}H^1(\Gamma) 
\end{align*}
where we made  $R$ and $\D_{10}$ large enough. Note $R$ is independent of $C$ above, which is important since  $C$ above grows as $\epsilon_0 \to 0$, and $\epsilon_0$ depends on $R$.
Hence
\begin{align*}
H^1(\Gamma)
\leq 
C\diam(\Kset) 
+
C_{\ourpower}\sum\limits_{j\in\N}\sum\limits_{P\in \Delta_j} \beta_\H(P,\D_{10} 2^{-j})^{\ourpower}2^{-j}
\end{align*}
which is bounded in turn by  a constant multiple (dependent on $\ourpower$) of \eqref{e:power-p-upper-bound}.

\end{document}